\documentclass[11pt,oneside,french]{amsart}
\usepackage{amssymb,amsfonts,amscd,mathrsfs,eurosym}
\usepackage{amsmath}
\usepackage{aecompl}
\usepackage{textcomp} 
\usepackage{stmaryrd}
\usepackage{enumerate}
\usepackage{enumitem}
\usepackage{bm}
\usepackage[T1]{fontenc}
\newcommand{\nc}{\newcommand}
\nc{\G}{{\Gamma}}
\nc{\BC}{{\mathbb C}}
\nc{\BQ}{{\mathbb Q}}
\nc{\BR}{{\mathbb R}}
 \nc{\BZ}{{\mathbb Z}}
 \nc{\BP}{{\mathbb P}}
\nc{\BN}{{\mathbb N}}
\nc{\BM}{{\mathbb M}}
\nc{\fH}{{\mathfrak{H}}}
\nc{\PSL}{{\mbox{PSL}_2(\BR)}}
\nc{\PS}{{\mbox{PSL}_2(\BZ)}}
\nc{\SL}{{\mbox{SL}_2(\BZ)}}
\nc{\SSL}{{\mbox{SL}_2(\BR)}}
\nc{\PSLL}{{\mbox{PSL}_2(\BC)}}
\nc{\s}{{\mathfrak{sl}(2)}}
\nc{\GL}{{\mbox{GL}_2^+(\BQ)}}
\newtheorem{numbered}{}[section]
\newtheorem{thm}[numbered]{Th\'eor\`eme}

\newtheorem{lem}[numbered]{Lemme}
\newtheorem{rem}[numbered]{Remarque}

\newtheorem{prop}[numbered]{Proposition}
\newtheorem{cor}[numbered]{Corollaire}

\numberwithin{equation}{section}

\allowdisplaybreaks  % un 'align' sur plusieurs pages

\begin{document}
\title[]{ Sur deux formules de Frobenius et Stickelberger\\ et inversion de Lagrange  }
\author[]{Roger Gay \& Marcel Grang\'e \& Ahmed Sebbar}
\address{Institut de Math\'ematiques de Bordeaux,
Universit\'e Bordeaux 1, 351 cours de la Lib\'eration F-33405
Talence cedex}
\email{Roger.Gay@math.u-bordeaux1.fr}\;  \email{mgrange@math.u-bordeaux1.fr} \; \email{ahmed.sebbar@math.u-bordeaux1.fr }

\begin{abstract}
 Nous pr\'esontons une preuve et une extension de deux formules de Frobenius et Stickelberger ainsi que des d\'eveloppements bas\'es sur la formule d'inversion de Lagrange.
\end{abstract}

 \maketitle

Le point de d\'epart du pr\'esent travail est les deux formules, utilis\'ees par Frobenius et Stickelberger dans leur important travail sur les fonctions elliptiques \cite{F1}. Ces deux formules
rel\`event du calcul diff\'erentiel pur et s'\'enoncent ainsi: Soient $U,V$ deux fonctions $n$-fois  contin\^ument diff\'erentiable sur un intervalle $J\subset \BR$, on a:
\begin{equation} \label{1}
D^{(n)} (V)= \sum_{0 \leq p \leq n} {n \atopwithdelims() p}  \frac{1}{p+1} D^p (U^{p+1}) D^{(n-p)} \left(\frac{V}{U^{p+1}} \right)
\end{equation}
et
\begin{equation} \label{2}
D^{(n)}(V)= \sum_{0 \leq p \leq n} (-1)^p {n \atopwithdelims() p}  \frac{1}{p+1} U^{-p-1} D^{(n-p)} (VD^p U^{p+1} ).
\end{equation}

Le calcul diff\'erentiel comporte diverses formules int\'eressantes \cite{F0}, notamment faisant
intervenir le produit de fonctions. La plus connue est la formule
de Leibniz donnant la d\'eriv\'ee d'ordre $n$ d'un produit. La formule d'inversion de Lagrange, \cite{F3} et \cite{F5}, poss\`ede de nombreuses applications dont la plus embl\'ematique est  la fonction arbre
\[a(x)= \sum_{n=1}^{\infty} n^{n-1}\frac{x^n}{n!}   \]
qui r\'esout l'\'equation $ \displaystyle a(x)= xe^{a(x)}   $. Cette fonction a beaucoup d'applications combinatoires et est souvent donn\'ee \`a l'aide de la classique fonction de Weber 
\[\displaystyle W(x)= -a(-x),\, W(x)e^{W(x)}=x.  \]   De la formule
d'inversion de Lagrange  on a pu d\'eduire ce qu'il
est convenu d'appeler la  formule du produit de Lagrange,
qui g\'en\'eralise la formule de Leibniz. Par ailleurs Frobenius et Stickelberger
indiquent les \'egalit\'es \eqref{1} et \eqref{2} (en bas de page de \cite{F1}, sans d\'emonstration) et d'autres  
 dont l'aspect rappelle encore la formule de Leibniz, sans toutefois
pouvoir se r\'eduire \`a cette derni\`ere. 

Le travail pr\'esent\'e ici propose de d\'emontrer ces diverses formules
en utilisant un m\^eme outil: les applications bilin\'eaires
$\Phi$, introduites et \'etudi\'ees dans la premi\`ere section.
La deuxi\`eme section est d\'evolue \`a deux g\'en\'eralisations des formules
de Frobenius-Stickelberger \cite{F1} et une application aux
fonctions enti\`eres de type exponentiel. La troisi\`eme section reprend
donc la formule du produit de Lagrange \`a partir des applications bilin\'eaires
$\Phi$, et gr\^ace \`a des calculs alg\'ebriques sur des fonctions et leurs
d\'eriv\'ees successives, \'evidemment sans utiliser la formule
d'inversion de Lagrange. Enfin la derni\`ere et quatri\`eme section traite
d'un r\'esultat de P. J. Olver qui a \'et\'e, nous semble t-il, succinctement avanc\'e dans
\cite{F4}, qui est ici int\'egralement d\'emontr\'e, notamment \`a
l'aide de la formule du produit de Lagrange. Ainsi on constate que
rien de ce qui est avanc\'e ici dans le cadre de la variable complexe,
ne d\'epend de la formule int\'egrale de Cauchy, contrairement \`a \cite{F5}: seule la th\'eorie des
s\'eries enti\`eres est utilis\'ee. 

Les concepts et notations sont assez courants, toutefois il est peut-\^etre utile 
de donner les pr\'ecisions qui suivent. Tous les intervalles consid\'er\'es
dans cette \'etude contiennent au mois deux points distincts. La fonction
d\'eriv\'ee d'une fonction complexe $f$ d\'efinie et d\'erivable sur un intervalle
est not\'ee $Df$, et les \'eventuelles d\'eriv\'ees successives sont not\'ees
$D^{k}f$. Un espace vectoriel complexe est aussi un espace vectoriel
r\'eel, et on d\'esigne par $E_{n}$ l'espace de Banach r\'eel
des fonctions complexes de classe $\mathcal{C}^{n}$
sur le segment r\'eel $\bigl[a,b\bigr]$. La norme \'etant : 
\[
\left\Vert f\right\Vert =\sum_{j=0}^{n}\left\Vert D^{j}f\right\Vert _{\left[a,b\right]}
\]
et par $\mathcal{U}_{n}$ l'ouvert de $E_{n}$ 
\[
\mathcal{U}_{n}=\bigl\{ f\in E_{n}\;;\;\forall  t\in\bigl[a,b\bigr],\; f\left(t\right)\neq0\bigr\}.
\]
Cet ouvert  est aussi  connexe car les fonctions de $E_{n}$ sont \`a valeurs
complexes. \'Etant donn\'ee une application $\mathbf{F}$ d\'efinie et diff\'erentiable
sur un ouvert $\Omega$ d'un espace norm\'e r\'eel $E$, \`a valeurs dans
un espace norm\'e r\'eel, sa diff\'erentielle en un point $x$ de $\Omega$
est not\'ee dans ce contexte $\mathfrak{D}\mathbf{F}\left(x\right)$
de mani\`ere \`a \'eviter la confusion avec l'op\'eration $D$ ci-dessus d\'ecrite. 

Pour tout couple $\left(n,p\right)\in \BN \times \BN^*$ on d\'esigne par $E\left(n,p\right)$
l'ensemble 
\[
\bigl\{\alpha=\left(\alpha_{1},\ldots,\alpha_{p}\right)\in\mathbb{N}^{p}\;;\;\alpha_{1}+\cdots+\alpha_{p}=n\bigr\}.
\]
Si le couple d'entiers v\'erifie en outre la condition $0\leq p\leq n$, on d\'esigne, comme d'habitude, par $\dbinom{n}{p}$
le nombre entier $\dfrac{n!}{p!\left(n-p\right)!}$, qui est inf\'erieur
ou \'egal \`a $2^{n}$ en vertu de la formule du bin\^ome. L'ensemble $E\left(n,p\right)$
est de cardinal $\dbinom{n+p-1}{p-1}$. 

Les notations de la quatri\`eme section, plus sp\'ecifiques \`a celle-ci,
sont rappel\'ees ou introduites au d\'ebut de cette derni\`ere section. 

\section{Les applications bilin\'eaires $\Phi$ }

\'Etant donn\'es un entier naturel $n$ et un intervalle $J\subset \BR$,
on consid\`ere une fonction complexe $u$ de classe $\mathcal{C}^{n}$
 et ne s'annulant pas sur $J$. \`A cette fonction $u$ est attach\'ee
l'application bilin\'eaire $\Phi_{n,u}$ d\'efinie sur l'espace vectoriel
complexe des fonctions de classe $\mathcal{C}^{n}$ sur 
$J$, \`a valeurs dans l'espace vectoriel complexe des fonctions continues
sur $J$: 
\begin{equation}\label{eq}
\Phi_{n,u}\left(f,g\right)=\sum_{p=0}^{n}\dbinom{n}{p}D^{p}\left(u^{p}f\right)D^{n-p}\left(u^{-p}g\right).
\end{equation}

\begin{thm}\label{t1}
Pour tout couple $\left(n,u\right)$ comme ci-dessus et pour tout
entier naturel $q$ on a: 
\begin{equation}\label{2.1}
\Phi_{n,u}\left(u^{q},u^{-q}\right)=\Phi_{n,u}\left(1,1\right).
\end{equation}
Pour tout couple $\left(n,u\right)$ comme ci-dessus on a: 
\begin{equation}
\Phi_{n,u}\left(f,g\right)=\sum_{r=0}^{n}\binom{n}{r}\Phi_{n-r,u}\left(1,1\right)D^{r}\left(fg\right)\label{eq2}.
\end{equation}
\end{thm}
\begin{proof}
\noindent En d\'eveloppant $D^{p}\left(u^{p}f\right)$ et $D^{n-p}\left(u^{-p}g\right)$
par la formule de Leibniz, apr\`es avoir interverti la sommation ${\displaystyle \sum_{p=0}^{n}}$
avec les sommations ${\displaystyle \sum_{0\leq j\leq p}}$ et ${\displaystyle \sum_{0\leq k\leq n-p}}$
de la formule de Leibniz, on obtient: 
\[
\Phi_{n,u}\left(f,g\right)={\displaystyle \sum_{0\leq j,k\leq n}Q_{n}\left(u,j,k\right)D^{j}f\, D^{k}g}
\]
o\`u l'on a pos\'e: 
\[
Q_{n}\left(u,j,k\right)=\left\{ \begin{array}{ccc}
0 & \; si\;\: j+k>n\\
\\
{\displaystyle \sum_{j\leq p\leq n-k}\binom{n}{p}\binom{p}{j}\binom{n-p}{k}}D^{p-j}u^{p}\, D^{n-k-p}u^{-p} & \; si\;\: j+k\leq n.
\end{array}\right.
\]
Par suite on obtient: 
\[
\Phi_{n,u}\left(f,g\right)=\sum_{r=0}^{n}\biggl(\,\sum_{j+k=r}Q_{n}\left(u,j,k\right)D^{j}f\, D^{k}g\biggr).
\]
Or, notant $r=j+k$, et supposant $r\leq n$, on a l'\'egalit\'e: 
\[
\binom{n}{p}\binom{p}{j}\binom{n-p}{k}=\binom{n}{r}\binom{r}{j}\binom{n-r}{p-j},
\]
d'o\`u l'on tire: 
\begin{align*}
Q_{n}\left(u,j,k\right)&= \binom{n}{r}\binom{r}{j}\sum_{q=0}^{n-r}\binom{n-r}{q}D^{q}u^{q+j}\, D^{n-r-q}u^{-q-j}\\
&= \binom{n}{r}\binom{r}{j}\Phi_{n-r,u}\left(u^{j},u^{-j}\right).
\end{align*}
En cons\'equence:
\begin{equation}\label{eq3}
\Phi_{n,u}\left(f,g\right)=\sum_{r=0}^{n}\binom{n}{r}\biggl(\,\sum_{j+k=r}\binom{r}{j}\Phi_{n-r,u}\left(u^{j},u^{-j}\right)D^{j}f\, D^{k}g\biggr).
\end{equation}

\noindent {\em Preuve de la formule \eqref{2.1}}: Pour tout entier naturel $q$ on a imm\'ediatement: 
\[
\Phi_{0,u}\left(u^{q},u^{-q}\right)=u^{q}u^{-q}=1=\Phi_{0,u}\left(1,1\right).
\]
Soit un entier naturel $n$ sup\'erieur ou \'egal \`a $1$, et supposons
la propri\'et\'e  vraie jusqu'au rang $n-1$. Compte-tenu de la formule \eqref{2.1},
pour tout entier naturel $q$ on a: 
\[
\Phi_{n,u}\left(u^{q},u^{-q}\right)=\sum_{r=0}^{n}\binom{n}{r}\biggl(\,\sum_{j+k=r}\binom{r}{j}\Phi_{n-r,u}\left(u^{j},u^{-j}\right)D^{j}u^{q}\, D^{k}u^{-q}\biggr).
\]
Puis, en vertu de l'hypoth\`ese de r\'ecurrence
\[
\Phi_{n,u}\left(u^{q},u^{-q}\right)=\Phi_{n,u}\left(1,1\right)u^{q}u^{-q}+\sum_{r=1}^{n}\binom{n}{r}\Phi_{n-r,u}\left(1,1\right)\biggl(\,\sum_{j+k=r}\binom{r}{j}D^{j}u^{q}\, D^{k}u^{-q}\biggr)
\]
\[
=\sum_{r=0}^{n}\binom{n}{r}\Phi_{n-r,u}\left(1,1\right)\biggl(\,\sum_{j+k=r}\binom{r}{j}D^{j}u^{q}\, D^{k}u^{-q}\biggr)=\sum_{r=0}^{n}\binom{n}{r}\Phi_{n-r,u}\left(1,1\right)D^{r}\left(u^{q}u^{-q}\right),
\]
ce qu'il fallait montrer.

\noindent {  \em Preuve de la formule \eqref{eq2}}: En vertu des formules \eqref{eq3} et \eqref{2.1} et de la formule
de Leibniz, on obtient: 
\[
\Phi_{n,u}\left(f,g\right)=\sum_{r=0}^{n}\binom{n}{r}\biggl(\,\sum_{j+k=r}\binom{r}{j}\Phi_{n-r,u}\left(1,1\right)D^{j}f\, D^{k}g\biggr)=\sum_{r=0}^{n}\binom{n}{r}\Phi_{n-r,u}\left(1,1\right)D^{r}\left(fg\right).
\]
\end{proof}
\begin{rem} \label{r1}
Pour tout couple de couples $\left(f_{1},g_{1}\right)$ et $\left(f_{2},g_{2}\right)$
de fonctions complexes de classe $\mathcal{C}^{n}$ et v\'erifiant $f_{1}g_{1}=f_{2}g_{2}$
on a: 
\[
\Phi_{n,u}\left(f_{1},g_{1}\right)=\Phi_{n,u}\left(f_{2},g_{2}\right)=\Phi_{n,u}\left(1,f_{1}g_{1}\right).
\]
\end{rem}
\section{Une extension des formules de Frobenius-Stickelberger }
\subsection{Les formules}
La proposition suivante \'etablit une extension de l'identit\'e \eqref{1} de Frobenius et Stickelberger \cite{F1}
\begin{prop}\label{t2}
Pour tout nombre entier $\lambda$ sup\'erieur ou \'egal \`a $1$, pour
tout triplet $\left(u,v,w\right)$ de fonctions complexes de classe
$\mathcal{C}^{n}$ sur un intervalle $J$, la fonction $u$ ne s'annulant
pas, on a la formule: 
\begin{equation}\label{eq4}
\sum_{p=0}^{n}\dbinom{n}{p}\dfrac{1}{p+\lambda}D^{p}\left(u^{p+\lambda}w\right)D^{n-p}\left(u^{-p-\lambda}v\right)=\sum_{p=0}^{n}\dbinom{n}{p}\dfrac{1}{p+\lambda}D^{p}\left(w\right)D^{n-p}\left(v\right).
\end{equation}
\end{prop}
\begin{proof} Il suffit de d\'emontrer le r\'esultat lorsque l'intervalle $J$ est un
segment $\bigl[a,b\bigr]$. 

\'Etant donn\'e un couple $\left(v,w\right)$
de fonctions de l'espace de Banach $E_{n}$, on consid\`ere l'application
\begin{equation*}
\begin{split}
\mathbf{G}_{n,v,w}&:\mathcal{U}_{n}\longrightarrow E_{0}\\
\mathbf{G}_{n,w,v}\left(u\right)&=\sum_{p=0}^{n}\dbinom{n}{p}\dfrac{1}{p+\lambda}D^{p}\left(u^{p+\lambda}w\right)D^{n-p}\left(u^{-p-\lambda}v\right).
\end{split}
\end{equation*}
Or, pour tout entier $\alpha$, l'application $u\longmapsto u^{\alpha}$
est diff\'erentiable sur l'ouvert $\mathcal{U}_{n}$, \`a valeurs dans
$E_{n}$, et sa diff\'erentielle est $h\longmapsto\alpha\, u^{\alpha-1}h$.
Par ailleurs, pour tout entier naturel $q\leq n$ l'application lin\'eaire
$D^{q}$ est continue de $E_{n}$ dans $E_{0}$. Donc l'application
$\mathbf{G}_{n,v,w}$ est diff\'erentiable sur l'ouvert $\mathcal{U}_{n}$
et on a: 
\[
\mathfrak{D}\mathbf{G}_{n,v,w}\left(u\right)=
\sum_{p=0}^{n}\dbinom{n}{p}\Bigl(D^{p}\left(u^{p+\lambda-1}hw\right)D^{n-p}\left(u^{-p-\lambda}v\right)-D^{p}\left(u^{p+\lambda}w\right)D^{n-p}\left(u^{-p-\lambda-1}hv\right)\Bigr).
\]
Introduisant l'application bilin\'eaire $\Phi_{n,u}$, consid\'er\'ee d\'efinie
sur $E_{m}\times E_{m}$ \`a valeurs dans $E_{0}$, on observe l'\'egalit\'e: 
\[
\mathfrak{D}\mathbf{G}_{n,v,w}\left(u\right)\cdot h=\Phi_{n,u}\left(hu^{\lambda-1}w,u^{-\lambda}v\right)-\Phi_{n,u}\left(u^{\lambda}w,hu^{-\lambda-1}v\right)
\]
soit, compte-tenu du th\'eor\`eme \ref{t1}: 
\[
\mathfrak{D}\mathbf{G}_{n,v,w}\left(u\right)\cdot h=\Phi_{n,u}\left(1,hu^{-1}vw\right)-\Phi_{n,u}\left(1,hu^{-1}vw\right)=0.
\]
L'application $\mathbf{G}_{n,v,w}$ est donc constante sur l'ouvert connexe 
$\mathcal{U}_{n}$. Mais il est clair qu'on a: 
\[
\mathbf{G}_{n,v,w}\left(1\right)=\sum_{p=0}^{n}\dbinom{n}{p}\dfrac{1}{p+\lambda}D^{p}\left(w\right)D^{n-p}\left(v\right)
\]
 l'application $\mathbf{G}_{n,v,w}$ est ainsi contante de
valeur ${\displaystyle \sum_{p=0}^{n}\dbinom{n}{p}\dfrac{1}{p+\lambda}D^{p}\left(w\right)D^{n-p}\left(v\right)}$. 
\end{proof}
Par dualit\'e, \`a partir de la formule de la proposition \ref{t2} ci-dessus,
on obtient une seconde formule qui \'etend l'identit\'e \eqref{2} de Frobenius-Stickelberger.
La proposition suivante pr\'ecise cette deuxi\`eme formule, et en avance
aussi une troisi\`eme, diff\'erente malgr\'e les apparences. 
\begin{prop}
Pour tout nombre entier $\lambda\geq1$, pour tout triplet $\left(u,v,w\right)$
de fonctions complexes de classe $\mathcal{C}^{n}$ sur un intervalle
$J$, la fonction $u$ ne s'annulant pas, on a:
\begin{equation}
\sum_{p=0}^{n}\dbinom{n}{p}\dfrac{\left(-1\right)^{p}}{p+\lambda}u^{-p-\lambda}D^{n-p}\left(v\, D^{p}\left(u^{p+\lambda}w\right)\right)= 
\sum_{p=0}^{n}\dbinom{n}{p}\dfrac{\left(-1\right)^{p}}{p+\lambda}D^{n-p}\left(v\, D^{p}w\right)\label{eq5}.
\end{equation}
\begin{equation}
\sum_{p=0}^{n}\dbinom{n}{p}\dfrac{\left(-1\right)^{p}}{p+\lambda}u^{p+\lambda}D^{p}\left(v\, D^{n-p}\left(u^{-p-\lambda}w\right)\right)= 
\sum_{p=0}^{n}\dbinom{n}{p}\dfrac{\left(-1\right)^{p}}{p+\lambda}D^{p}\left(v\, D^{n-p}w\right).\label{eq6}
\end{equation}
\end{prop}
\begin{proof}
Seule est donn\'ee une esquisse de la preuve de la formule \eqref{eq5}. Pour toute fonction
$\varphi$ de classe $\mathcal{C}^{\infty}$ \`a support compact dans
l'intervalle $J$, par $n-p$ int\'egrations par parties et en vertu
de la formule \eqref{eq4} on a: 
\[
\int_{J}\varphi\sum_{p=0}^{n}\dbinom{n}{p}\dfrac{\left(-1\right)^{p}}{p+\lambda}u^{-p-\lambda}D^{n-p}\left(v\, D^{p}\left(u^{p+\lambda}w\right)\right)\, dt=\left(-1\right)^{n}\int_{J}v\sum_{p=0}^{n}\dbinom{n}{p}\dfrac{1}{p+\lambda}D^{p}wD^{n-p}\varphi\, dt
\]
\`A nouveau $n-p$ int\'egrations par parties conduisent \`a l'\'egalit\'e 
\[
\left(-1\right)^{n}\int_{J}v\sum_{p=0}^{n}\dbinom{n}{p}\dfrac{1}{p+\lambda}D^{p}wD^{n-p}\varphi\, dt=\int_{J}\varphi\sum_{p=0}^{n}\dbinom{n}{p}\dfrac{\left(-1\right)^{p}}{p+\lambda}D^{n-p}\left(v\, D^{p}w\right) dt.
\]
Enfin on conclut gr\^ace au lemme de du Bois-Reymond. 
\end{proof}

\subsection{Une application de la premi\`ere formule de Frobenius-Stickelberger}
Choisissons les trois fonctions $u$, $v$ et $w$ comme suit: 
\[
u\left(t\right)=\exp\left(xt\right);\quad v\left(t\right)=\exp\left(yt\right);\quad w\left(t\right)=\exp\left(zt\right)
\]
o\`u $x$, $y$ et $z$ d\'esignent des nombres complexes. Pour tout 
entier $\lambda\geq1$ on obtient: 
\[
\sum_{p=0}^{n}\binom{n}{p}\dfrac{1}{p+\lambda}\left(z+\left(p+\lambda\right)x\right)^{p}\left(y-\left(p+\lambda\right)x\right)^{n-p}=\sum_{p=0}^{n}\binom{n}{p}\dfrac{1}{p+\lambda}z^{p}y^{n-p}.
\]
La diff\'erence des deux membres de cette relation est, pour $\left(x,y,z\right)$
fix\'e, une fonction \textit{fraction rationnelle} en $\lambda$, qui
s'annule sur l'ensemble infini $\mathbb{N}^{*}$. Donc cette fonction
fraction rationnelle est la fonction nulle et on obtient ainsi l'identit\'e:
\begin{equation}\label{eq51}
\sum_{p=0}^{n}\binom{n}{p}\dfrac{1}{p+\lambda}\left(z+\left(p+\lambda\right)x\right)^{p}\left(y-\left(p+\lambda\right)x\right)^{n-p}=\sum_{p=0}^{n}\binom{n}{p}\dfrac{1}{p+\lambda}z^{p}y^{n-p}
\end{equation}
 valable pour tout $\left(x,y,z\right)\in {\BC}^{3}$ et tout
$\lambda \in {\BC}\setminus\left\{ 0,-1,\ldots,-n\right\} $. 
\begin{prop}
Soit une fonction enti\`ere $f$ de type exponentiel. Il existe un voisinage
sym\'etrique convexe compact $A$ de $x=0$ dans $\mathbb{C}$ v\'erifiant: Pour tout $\left(x,y,z\right)$ de $A\times\mathbb{C}^{2}$ et tout
$\lambda$ de $\mathbb{C}\setminus\left\{ 0,-1,\ldots,-n,\ldots\right\} $
on a: 
\begin{equation}\label{eq6}
\sum_{p=0}^{\infty}\dfrac{1}{p+\lambda}\dfrac{\left(z+\left(p+\lambda\right)x\right)^{p}}{p!}D^{p}f\left(y-\left(p+\lambda\right)x\right)= 
\sum_{p=0}^{\infty}\dfrac{1}{p+\lambda}\dfrac{z^{p}}{p!}D^{p}f\left(y\right).
\end{equation}
Pour tout $\left(x,\lambda\right)$ de $A\times\mathbb{C}$ on a:
\begin{equation}\label{eq7}
f\left(\lambda x\right)=f\left(0\right)+\lambda\sum_{p=1}^{\infty}\left(\lambda-p\right)^{p-1}\dfrac{x^{p}}{p!}D^{p}f\left(px\right).
\end{equation}
Pour tout $x$ de $A$ et pour tout entier $m$ sup\'erieur ou \'egal
\`a $1$ on a: 
\begin{equation}\label{eq8}
\dfrac{1}{m!}D^{m}f\left(0\right)=\sum_{p=m}^{\infty}\dbinom{p-1}{m-1}\dfrac{\left(-1\right)^{p-m}p^{p-m}}{p!}x^{p-m}D^{p}f\left(px\right).
\end{equation}
\end{prop}
\begin{proof}
\noindent {\em D\'emontrons la formule \eqref{eq6}}: Par hypoth\`ese, il existe deux nombres strictement
positifs $C$ et $K$ tels que pour tout entier naturel $n$ on ait
l'in\'egalit\'e $\bigl|D^{n}f\left(0\right)\bigr|\leq K^{n}C$. De l\`a,
on effectue les majorations suivantes:
\begin{align*}
\sum_{n=0}^{\infty}\sum_{p=0}^{n}&\left|\dfrac{1}{n!}D^{n}f\left(0\right)\right|\dbinom{n}{p}\dfrac{1}{\left|p+\lambda\right|}\bigl|z+\left(p+\lambda\right)x\bigr|^{p}\bigl|y-\left(p+\lambda\right)x\bigr|^{n-p}
\\
&\leq  C\sum_{p=0}^{\infty}\dfrac{\bigl|z+\left(p+\lambda\right)x\bigr|^{p}}{p!\left|p+\lambda\right|}\left(\sum_{n=p}^{\infty}K^{n}\dfrac{\bigl|y-\left(p+\lambda\right)x\bigr|^{n-p}}{\left(n-p\right)!}\right)
\\
&\leq  C\sum_{p=0}^{\infty}K^{p}\dfrac{\bigl|z+\left(p+\lambda\right)x\bigr|^{p}}{p!\left|p+\lambda\right|}\exp\left(K\left|y-\left(p+\lambda\right)x\right|\right)
\\
&\leq  C\Bigl(\exp\left(K\left|y-\lambda x\right|\right)\Bigr)\sum_{p=0}^{\infty}K^{p}\dfrac{\bigl|z+\left(p+\lambda\right)x\bigr|^{p}}{p!\left|p+\lambda\right|}\left(\exp\left(K\left|x\right|\right)\right)^{p}.
\end{align*}
Dans le cas $x=0$, la s\'erie ci-dessus est, \'evidemment, convergente.
Dans le cas $x\neq0$, on peut \'ecrire: 
\[
\sum_{p=1}^{\infty}K^{p}\dfrac{\bigl|z+\left(p+\lambda\right)x\bigr|^{p}}{p!\left|p+\lambda\right|}\left(\exp\left(K\left|x\right|\right)\right)^{p}\leq\sum_{p=1}^{\infty}
\left(K^{p}\dfrac{p^{p}\left(\left|x\right|\exp\left(K\left|x\right|\right)\right)^{p}}{p!\left|p+\lambda\right|}\left(1+\dfrac{\left|z+\lambda x\right|}{p\left|x\right|}\right)^{p}\right).
\]
Gr\^ace \`a l'in\'egalit\'e de Stirling : $p!\geq p^{p}e^{-p}\sqrt{2\pi p}$,
on obtient la majoration: 
\[
\sum_{p=1}^{\infty}\left(K^{p}\dfrac{p^{p}\left(\left|x\right|\exp\left(K\left|x\right|\right)\right)^{p}}{p!\left|p+\lambda\right|}
\left(1+\dfrac{\left|z+\lambda x\right|}{p\left|x\right|}\right)^{p}\right)\lesssim\exp\left(\dfrac{\left|z+\lambda x\right|}{\left|x\right|}\right)
\sum_{p=1}^{\infty}\left(eK\left|x\right|\exp\left(K\left|x\right|\right)\right)^{p}p^{-\tfrac{3}{2}}.
\]
En cons\'equence, dans le voisinage sym\'etrique convexe compact 
\[
A=\bigl\{ x\in \BC, \,eK\left|x\right|\exp K\left|x\right|\leq1\bigr\}
\]
du point $x=0$, la s\'erie double 
\[
\sum_{n=0}^{\infty}\left(\sum_{p=0}^{n}\dfrac{1}{n!}D^{n}f\left(0\right)\dbinom{n}{p}\dfrac{1}{p+\lambda}\left(z+\left(p+\lambda\right)x\right)^{p}\left(y-\left(p+\lambda\right)x\right)^{n-p}\right)
\]
est absolument convergente, et en intervertissant les sommes, elle
s'\'ecrit d'une part: 
\[
\sum_{p=0}^{\infty}\dfrac{1}{p+\lambda}\dfrac{\left(z+\left(p+\lambda\right)x\right)^{p}}{p!}D^{p}f\left(y-\left(p+\lambda\right)x\right).
\]
D'autre part, gr\^ace \`a l'identit\'e \eqref{eq51}, et en intervertissant les sommes
elle s'\'ecrit aussi: 
\[
\sum_{n=0}^{\infty}\left(\sum_{p=0}^{n}\dfrac{1}{n!}D^{n}f\left(0\right)\dbinom{n}{p}\dfrac{1}{p+\lambda}z^{p}y^{n-p}\right)=\sum_{p=0}^{\infty}\dfrac{1}{p+\lambda}\dfrac{z^{p}}{p!}D^{p}f\left(y\right)
\]
d'o\`u la conclusion. 

\noindent{\em D\'emontrons la formule \eqref{eq7}}: Faisant $z=0$ dans l'identit\'e \eqref{eq6} et multipliant
par $\lambda$, on obtient: 
\[
f\left(y\right)=f\left(y-\lambda x\right)+\lambda\sum_{p=1}^{\infty}\dfrac{\left(p+\lambda\right)^{p-1}}{p!}x^{p}D^{p}f\left(y-\left(p+\lambda\right)x\right).
\]
Mais le terme g\'en\'eral de cette s\'erie est major\'e comme suit: 
\begin{align*}
&\left|\dfrac{\left(p+\lambda\right)^{p-1}}{p!}x^{p}D^{p}f\left(y-\left(p+\lambda\right)x\right)\right|\\
&\leq C\Bigl(\exp\left(K\left|y-\lambda x\right|\right)\Bigr)K^{p}\dfrac{\left|p+\lambda\right|^{p-1}}{p!}\left(K\left|x\right|\exp\left(K\left|x\right|\right)\right)^{p}
\\
&\lesssim\exp\left(\left|\lambda\right|+K\left|y-\lambda x\right|\right)\left(K\left|x\right|\exp\left(K\left|x\right|\right)\right)^{p}p^{-\tfrac{3}{2}}.
\end{align*}
La s\'erie ci-dessus est donc normalement convergente par rapport \`a
$\lambda$ sur tout compact de $\mathbb{C}$. En cons\'equence la formule
ci-dessus exprimant $f\left(y\right)$ s'\'etend \`a tout triplet $\left(x,y,\lambda\right)$
appartenant \`a $A\times\mathbb{C}\times\mathbb{C}$. La formule \eqref{eq7}
s'obtient en faisant $y=\lambda x$ et en rempla\c{c}ant $\left(x,\lambda\right)$
par $\left(-x,-\lambda\right)$.

\noindent{\em D\'emontrons la formule \eqref{eq8}}: En d\'eveloppant $\left(\lambda-p\right)^{p-1}$
par la formule du bin\^ome, et gr\^ace aux majorations effectu\'ees pour
d\'emontrer \eqref{eq7}, il r\'esulte aussi que pour tout $\left(x,\lambda\right)$
appartenant \`a $A\times\mathbb{C}$: 
\[
f\left(\lambda x\right)=f\left(0\right)+\sum_{m=1}^{\infty}\left(-1\right)^{m}\lambda^{m}\left(\sum_{p=m}^{\infty}\dbinom{p-1}{m-1}\dfrac{\left(-1\right)^{p}p^{p-m}}{p!}x^{p-m}D^{p}f\left(px\right)\right).
  \qedhere
\]
\end{proof}
\section{Formule du produit de Lagrange}
Dans cette section et la suivante, pour tout couple $\left(\psi,f\right)$
de fonctions complexes de classe $\mathcal{C}^{n}$ sur un intervalle
$J$, la fonction d\'esign\'ee par $D^{m-1}\left(\psi^{m}D\left(f\right)\right)$ est, pour $m=0$,  la fonction $f=D^{-1}\left(Df\right)$. Mais tout
d'abord il convient de s'assurer du lemme suivant sur les applications
bilin\'eaires $\Phi$. 
\begin{lem}\label{lem1}
Pour tout couple $\left(m,N\right)$ d'entiers naturels et pour toute
fonction complexe $\psi$ de classe $\mathcal{C}^{n}$ sur un intervalle
$J$ et ne s'y annulant pas, on a: 
\begin{equation}\label{eq9}
\Phi_{m,\psi}\left(1,\psi^{N}D\psi\right)=\dfrac{1}{m+1}\Bigl(\Phi_{m+1,\psi}\left(1,\psi^{N+1}\right)-D^{m+1}\left(\psi^{N+1}\right)\Bigr)
\end{equation}
\begin{equation}\label{eq10}
\Phi_{m,\psi}\Bigl(1,\psi^{N}\left(D\psi\right)^{2}\Bigr)=\dfrac{1}{m+1}\Bigl(\Phi_{m+1,\psi}\left(1,\psi^{N+1}D\psi\right)-D^{m+1}\left(\psi^{N+1}D\psi\right)\Bigr)
\end{equation}
\[
=\dfrac{1}{\left(m+1\right)\left(m+2\right)}\Bigl(\Phi_{m+2,\psi}\left(1,\psi^{N+2}\right)-D^{m+2}\left(\psi^{N+2}\right)\Bigr)-\dfrac{1}{m+1}D^{m+1}\left(\psi^{N+1}D\psi\right).
\]
\end{lem}
\begin{proof}
{\em D\'emontrons la formule \eqref{eq9}}: Compte-tenu de la remarque \eqref{r1}  on a successivement: 
\begin{align*}
\Phi_{m,\psi}&\left(1,\psi^{N}D\psi\right)=\Phi_{m,\psi}\left(D\psi,\psi^{N}\right)=\sum_{p=0}^{m}\dbinom{m}{p}D^{p}\left(\psi^{p}D\psi\right)D^{m-p}\left(\psi^{N-p}\right)
\\
&=\sum_{p=0}^{m}\dbinom{m}{p}\dfrac{1}{p+1}D^{p+1}\left(\psi^{p+1}\right)D^{m-p}\left(\psi^{N-p}\right)
\\
&=\dfrac{1}{m+1}\sum_{p=0}^{m}\dbinom{m+1}{p+1}D^{p+1}\left(\psi^{p+1}\right)D^{m+1-\left(p+1\right)}\left(\psi^{N-p}\right)
\\
&=-\dfrac{1}{m+1}D^{m+1}\left(\psi^{N+1}\right)+\dfrac{1}{m+1}\sum_{q=0}^{m+1}\dbinom{m+1}{q}D^{q}\left(\psi^{q}\right)D^{m+1-q}\left(\psi^{N+1-q}\right)
\\
&=\dfrac{1}{m+1}\Bigl(\Phi_{m+1,\psi}\left(1,\psi^{N+1}\right)-D^{m+1}\left(\psi^{N+1}\right)\Bigr).
\end{align*}
{\em Montrons la formule \eqref{eq10}}: Compte-tenu de la remarque \eqref{r1}  on a successivement: 
\begin{align*}
\Phi_{m,\psi}&\Bigl(1,\psi^{N}\left(D\psi\right)^{2}\Bigr)= \Phi_{m,\psi}\left(D\psi,\psi^{N}D\psi\right)=\sum_{p=0}^{m}\dbinom{m}{p}D^{p}\left(\psi^{p}D\psi\right)D^{m-p}\left(\psi^{N-p}D\psi\right)
\\
&= \sum_{p=0}^{m}\dbinom{m}{p}\dfrac{1}{p+1}D^{p+1}\left(\psi^{p+1}\right)D^{m-p}\left(\psi^{N-p}D\psi\right)
\\
&= \dfrac{1}{m+1}\sum_{p=0}^{m}\dbinom{m+1}{p+1}D^{p+1}\left(\psi^{p+1}\right)D^{m+1-\left(p+1\right)}\left(\psi^{N-p}D\psi\right)
\\
&= -\dfrac{1}{m+1}D^{m+1}\left(\psi^{N+1}D\psi\right)+\dfrac{1}{m+1}\sum_{q=0}^{m+1}\dbinom{m+1}{q}D^{q}\left(\psi^{q}\right)D^{m+1-q}\left(\psi^{N+1-q}D\psi\right)
\\
&= \dfrac{1}{m+1}\Bigl(\Phi_{m+1,\psi}\left(1,\psi^{N+1}D\psi\right)-D^{m+1}\left(\psi^{N+1}D\psi\right)\Bigr)
\\
&= \dfrac{1}{\left(m+1\right)\left(m+2\right)}\Bigl(\Phi_{m+2,\psi}\left(1,\psi^{N+2}\right)-D^{m+2}\left(\psi^{N+2}\right)\Bigr)-\dfrac{1}{m+1}D^{m+1}\left(\psi^{N+1}D\psi\right)
\end{align*}
en vertu de la formule \eqref{eq9} ci-dessus d\'emontr\'ee.\end{proof}
\begin{thm}
Pour tout triplet $\left(\psi,f,g\right)$ de fonctions complexes
de classe $\mathcal{C}^{n}$ sur un intervalle $J$, on a la formule:
\begin{equation}
\dfrac{1}{n!}D^{n-1}\left(\psi^{n}D\left(fg\right)\right)=\sum_{k=0}^{n}\dfrac{1}{k!}D^{k-1}\left(\psi^{k}D\left(f\right)\right)\dfrac{1}{\left(n-k\right)!}D^{n-k-1}\left(\psi^{n-k}D\left(g\right)\right).\label{eq:3-1}
\end{equation}
 Plus g\'en\'eralement, pour entier naturel $p$ sup\'erieur ou \'egal \`a $2$
et tout syst\`eme $\left(\psi,f_{1},\ldots,f_{p}\right)$ de fonctions
complexes de classe $\mathcal{C}^{n}$ sur un intervalle $J$, on
a la formule:
\[
\dfrac{1}{n!}D^{n-1}\left(\psi^{n}D\left(f_{1}\cdots f_{p}\right)\right)=
\sum_{\alpha_{1}+\cdots+\alpha_{p}=n}\left(\prod_{j=1}^{p}\dfrac{1}{\alpha_{j}!}D^{\alpha_{j}-1}\left(\psi^{\alpha_{j}}D\left(f_{j}\right)\right)\right).
\]
\end{thm}
\begin{proof}
La seconde formule s'obtient imm\'ediatement \`a partir de la premi\`ere,
par r\'ecurrence sur l'entier $p$. La d\'emonstration ne concerne donc
que la premi\`ere formule, qui est clairement vraie dans les cas $n=0$
et $n=1$ ; aussi l'entier $n$ est suppos\'e sup\'erieur ou \'egal \`a $2$ dans
ce qui suit. L'id\'ee consiste \`a exprimer le membre de droite de la formule de Lagrange
en fonction des d\'eriv\'ees successives $D^{j}\left(fg\right)$ o\`u $j$
appartient \`a $\left\{ 0,\ldots,n\right\} $. Si on suppose que la
fonction $\psi$ ne s'annule pas, on constate qu'apparaissent les
applications bilin\'eaires $\Phi$ introduites dans la section 2.

\noindent {\em R\'eduction au cas o\`u la fonction $\psi$ ne s'annule pas}:\\
\noindent Ainsi, on suppose la formule de Lagrange pour tout triplet $\left(\varphi,u,v\right)$
de fonctions complexes de classe $\mathcal{C}^{n}$ sur un intervalle
$I$ lorsque la fonction $\varphi$ ne s'annule pas. Soit un triplet
$\left(\psi,f,g\right)$ de fonctions complexes de classe $\mathcal{C}^{n}$
sur un intervalle $J$ et notons $F$ le sous-ensemble ferm\'e $\left\{ \psi=0\right\} $
de $J$. Gr\^ace \`a l'hypoth\`ese, la formule de Lagrange est acquise sur
l'int\'erieur de $F$ qui est une r\'eunion au plus d\'enombrable d'intervalles,
ouverts dans $J$, et pour la m\^eme raison sur le compl\'ementaire $\Omega$
de $F$. La formule de Lagrange est donc vraie en tout point de l'adh\'erence
$\overline{\Omega}$ dans $J$ de $\Omega$. Enfin, sachant qu'on
a $J=\overset{\circ}{F}\cup\overline{\Omega}$ la formule est vraie
sur tout l'intervalle $J$. On suppose dor\'enavant que la fonction $\psi$ ne s'annule
pas. Compte-tenu des relations 
\[
\psi^{k}Df=D\left(\psi^{k}f\right)-\left(k\psi^{k-1}D\psi\right)f,\quad\psi^{n-k}Dg=D\left(\psi^{n-k}g\right)-\left(\left(n-k\right)\psi^{n-k-1}D\psi\right)g
\]
le membre de droite de la formule de Lagrange s'\'ecrit comme la somme
de quatre termes 
\[
T_{1}\left(f,g\right)+T_{2}\left(f,g\right)-T_{3}\left(f,g\right)-T_{4}\left(f,g\right)
\]
qui sont exprim\'es et trait\'es dans ce qui suit.
%\item 

\noindent {\em Expression du premier terme}:
\begin{align*}
&T_{1}\left(f,g\right)=\sum_{k=0}^{n}\binom{n}{k}D^{k}\left(\psi^{k}f\right)D^{n-k}\left(\psi^{n-k}g\right)=\Phi_{n,\psi}\left(f,\psi^{n}g\right)
\\
&=\sum_{r=0}^{n}\binom{n}{r}\Phi_{n-r,\psi}\left(1,1\right)D^{r}\left(\psi^{n}fg\right)=\sum_{r=0}^{n}\binom{n}{r}\left(\sum_{j=0}^{r}\binom{r}{j}D^{r-j}\left(\psi^{n}\right)D^{j}\left(fg\right)\right)\Phi_{n-r,\psi}\left(1,1\right)
\\
&=\sum_{j=0}^{n}\left(\sum_{r=j}^{n}\binom{n}{r}\binom{r}{j}\Phi_{n-r,\psi}\left(1,1\right)D^{r-j}\left(\psi^{n}\right)\right)D^{j}\left(fg\right)
\\
&=\sum_{j=0}^{n}\binom{n}{j}\left(\sum_{q=0}^{n-j}\binom{n-j}{q}\Phi_{n-j-q,\psi}\left(1,1\right)D^{q}\left(\psi^{n}\right)\right)D^{j}\left(fg\right)=
\sum_{j=0}^{n}\binom{n}{j}\Phi_{n-j,\psi}\left(1,\psi^{n}\right)D^{j}\left(fg\right).
\end{align*}

%\item 
\noindent {\em Expression du deuxi\`eme terme}: 
\begin{align*}
T_{2}\left(f,g\right)&= \sum_{k=1}^{n-1}\binom{n}{k}k\left(n-k\right)D^{k-1}\left(\psi^{k-1}fD\psi\right)D^{n-k-1}\left(\psi^{n-k}gD\psi\right)
\\
&= n\left(n-1\right)\sum_{k=1}^{n-1}\binom{n-2}{k-1}D^{k-1}\left(\psi^{k-1}fD\psi\right)D^{n-k-1}\left(\psi^{n-k}gD\psi\right)
\\
&= n\left(n-1\right)\sum_{j=0}^{n-2}\binom{n-2}{j}D^{j}\left(\psi^{j}fD\psi\right)D^{n-2-j}\left(\psi^{n-2-j}gD\psi\right)
\\
&= n\left(n-1\right)\Phi_{n-2,\psi}\left(fD\psi,\psi^{n-2}gD\psi\right)
\\
&= n\left(n-1\right)\sum_{r=0}^{n-2}\binom{n-2}{r}\Phi_{n-2-r,\psi}\left(1,1\right)D^{r}\Bigl(\psi^{n-2}fg\left(D\psi\right)^{2}\Bigr)
\\
&= n\left(n-1\right)\sum_{r=0}^{n-2}\binom{n-2}{r}\left(\sum_{j=0}^{r}\binom{r}{j}D^{r-j}\Bigl(\psi^{n-2}\left(D\psi\right)^{2}\Bigr)D^{j}\left(fg\right)\right)\Phi_{n-2-r,\psi}\left(1,1\right)
\\
&= n\left(n-1\right)\sum_{j=0}^{n-2}\left(\sum_{r=j}^{n-2}\binom{n-2}{r}\binom{r}{j}\Phi_{n-2-r,\psi}\left(1,1\right)D^{r-j}\Bigl(\psi^{n-2}\left(D\psi\right)^{2}\Bigr)\right)D^{j}\left(fg\right)
\\
&= n\left(n-1\right)\sum_{j=0}^{n-2}\binom{n-2}{j}\left(\sum_{q=0}^{n-2-j}\binom{n-2-j}{q}\Phi_{n-2-j-q,\psi}\left(1,1\right)D^{q}\Bigl(\psi^{n-2}\left(D\psi\right)^{2}\Bigr)\right)D^{j}\left(fg\right)
\\
&= n\left(n-1\right)\sum_{j=0}^{n-2}\binom{n-2}{j}\Phi_{n-2-j,\psi}\Bigl(1,\psi^{n-2}\left(D\psi\right)^{2}\Bigr)D^{j}\left(fg\right).
\end{align*}
%\item 
\noindent {\em Expression du troisi\`eme terme}:
\begin{align*}
T_{3}\left(f,g\right)&= \sum_{k=1}^{n}\binom{n}{k}kD^{k-1}\left(\psi^{k-1}fD\psi\right)D^{n-k}\left(\psi^{n-k}g\right)
\\
&= n\sum_{k=1}^{n}\binom{n-1}{k-1}D^{k-1}\left(\psi^{k-1}fD\psi\right)D^{n-k}\left(\psi^{n-k}g\right)
\\
&= n\sum_{j=0}^{n-1}\binom{n-1}{j}D^{j}\left(\psi^{j}fD\psi\right)D^{n-1-j}\left(\psi^{n-1-j}g\right)=n\Phi_{n-1,\psi}\left(fD\psi,\psi^{n-1}g\right)
\\
&= n\sum_{r=0}^{n-1}\binom{n-1}{r}\Phi_{n-1-r,\psi}\left(1,1\right)D^{r}\left(\psi^{n-1}fgD\psi\right)
\\
&= n\sum_{r=0}^{n-1}\binom{n-1}{r}\left(\sum_{j=0}^{r}\binom{r}{j}D^{r-j}\left(\psi^{n-1}D\psi\right)D^{j}\left(fg\right)\right)\Phi_{n-1-r,\psi}\left(1,1\right)
\\
&= n\sum_{j=0}^{n-1}\left(\sum_{r=j}^{n-1}\binom{n-1}{r}\binom{r}{j}\Phi_{n-1-r,\psi}\left(1,1\right)D^{r-j}\left(\psi^{n-1}D\psi\right)\right)D^{j}\left(fg\right)
\\
&= n\sum_{j=0}^{n-1}\binom{n-1}{j}\left(\sum_{q=0}^{n-1-j}\binom{n-1-j}{q}\Phi_{n-1-j-q,\psi}\left(1,1\right)D^{q}\left(\psi^{n-1}D\psi\right)\right)D^{j}\left(fg\right)
\\
&= n\sum_{j=0}^{n-1}\binom{n-1}{j}\Phi_{n-1-j,\psi}\left(1,\psi^{n-1}D\psi\right)D^{j}\left(fg\right).
\end{align*}
%\item 

% test
%
% \begin{align}
%  A &= B    \notag \\
%  C &= B  \label{plustard}
%  \end{align}
% 
% end test

\noindent {\em Expression du quatri\`eme terme}:
\[
T_{4}\left(f,g\right)=\sum_{k=0}^{n-1}\binom{n}{k}\left(n-k\right)D^{k}\left(\psi^{k}f\right)D^{n-k-1}\left(\psi^{n-k-1}gD\psi\right)
\]
\[
=\sum_{p=1}^{n}\binom{n}{p}pD^{p-1}\left(\psi^{p-1}gD\psi\right)D^{n-p}\left(\psi^{n-p}f\right)=\Sigma_{3}\left(g,f\right)=\Sigma_{3}\left(f,g\right).
\]
En cons\'equence le membre de droite de la formule de Lagrange s'\'ecrit
comme suit: 
\[
\sum_{j=0}^{n}C_{n}\left(j,\psi\right)D^{j}\left(fg\right)
\]
o\`u les fonctions coefficients $C_{n}\left(j,\psi\right)$ sont donn\'ees
par: 
\[
C_{n}\left(n,\psi\right)=\Phi_{0,\psi}\left(1,\psi^{n}\right)
\]
\[
C_{n}\left(n-1,\psi\right)=n\Phi_{1,\psi}\left(1,\psi^{n}\right)-2n\Phi_{0,\psi}\left(1,\psi^{n-1}D\psi\right)
\]
et pour tout entier naturel $j\leq n-2$: 
\[
C_{n}\left(j,\psi\right)=\dbinom{n}{j}\Phi_{n-j,\psi}\left(1,\psi^{n}\right)+n\left(n-1\right)\dbinom{n-2}{j}\Phi_{n-2-j,\psi}\Bigl(1,\psi^{n-2}\left(D\psi\right)^{2}\Bigr)
\]
\[
-2n\dbinom{n-1}{j}\Phi_{n-1-j,\psi}\left(1,\psi^{n-1}D\psi\right).
\]
Comme on a en g\'en\'eral $\displaystyle \Phi_{0,\psi}\left(F,G\right)=FG$ \,et\, $\displaystyle \Phi_{1,\psi}\left(F,G\right)=D\left(FG\right)+\dfrac{D\psi}{\psi}FG$,
il vient: 
\[
C_{n}\left(n,\psi\right)=\psi^{n}
\]
\[
C_{n}\left(n-1,\psi\right)=n\left(n-1\right)\psi^{n-1}D\psi=\left(n-1\right)D\left(\psi^{n}\right)
\]
et pour tout entier naturel $j\leq n-2$, en vertu du lemme \eqref{lem1} pr\'ec\'edent: 
\[
C_{n}\left(j,\psi\right)=\dbinom{n}{j}\Phi_{n-j,\psi}\left(1,\psi^{n}\right)+n\left(n-1\right)\dbinom{n-2}{j}\dfrac{1}{\left(n-j-1\right)\left(n-j\right)}\Phi_{n-j,\psi}\left(1,\psi^{n}\right)
\]
\[
-n\left(n-1\right)\dbinom{n-2}{j}\dfrac{1}{\left(n-j-1\right)\left(n-j\right)}D^{n-j}\left(\psi^{n}\right)
\]
\[
-n\left(n-1\right)\dbinom{n-2}{j}\dfrac{1}{n-j-1}D^{n-j-1}\left(\psi^{n-1}D\psi\right)
\]
\[
-2n\dbinom{n-1}{j}\dfrac{1}{n-j}\Phi_{n-j,\psi}\left(1,\psi^{n}\right)+2n\dbinom{n-1}{j}\dfrac{1}{n-j}D^{n-j}\left(\psi^{n}\right)
\]
\[
=\left\{ \begin{array}{cc}
0 & \quad si\quad j=0\\
\\
\dbinom{n-1}{j-1}D^{n-j}\left(\psi^{n}\right) & \quad si\quad1\leq j\leq n-2
\end{array}\right.
\]
En cons\'equence, on aboutit \`a: 
\[
\sum_{k=0}^{n}\binom{n}{k}D^{k-1}\Bigl(\psi^{k}D\left(f\right)\Bigr)D^{n-k-1}\left(\psi^{n-k}D\left(g\right)\right)=\sum_{j=0}^{n}C_{n}\left(j,\psi\right)D^{j}\left(fg\right).
\]
\[
=\sum_{j=1}^{n}\dbinom{n-1}{j-1}D^{n-j}\left(\psi^{n}\right)D^{j}\left(fg\right)=\sum_{k=0}^{n-1}\dbinom{n-1}{k}D^{n-1-k}\left(\psi^{n}\right)D^{k+1}\left(fg\right)=D^{n-1}\left(\psi^{n}D\left(fg\right)\right)
\]
%\end{itemize}
\end{proof}
\begin{cor}\label{t10}
Pour tout couple $\left(\psi,f\right)$ de fonctions complexes de
classe $\mathcal{C}^{n}$ sur un intervalle $J$, et pour tout polyn\^ome
$P\in\mathbb{C}\left[X\right]$ on a la formule:
\begin{equation}\label{eq11}
P\left(\sum_{n=0}^{\infty}\dfrac{1}{n!}D^{n-1}\left(\psi^{n}Df\right)X^{n}\right)=\sum_{n=0}^{\infty}\dfrac{1}{n!}D^{n-1}\left(\psi^{n}D\left(P\left(f\right)\right)\right)X^{n}.
\end{equation}
\end{cor}

\subsection*{Une autre formule de Frobenius-Stickelberger}
Dans ce paragraphe la fonction complexe $\psi$ est de
classe $\mathcal{C}^{\infty}$ sur un segment $J$. Frobenius et
Stickelberger \cite{F1} ont avanc\'e une autre formule qu'il
est possible d'\'etablir \`a l'aide de la formule du produit de Lagrange,
en rempla\c{c}ant les fonctions $f$ et $g$ par la fonction $\psi^{-1}$
si la fonction $\psi$ ne s'annule pas, et plus g\'en\'eralement on peut
remplacer les fonctions $f$ et $g$ par la fonction $\psi^{-p}$
ou $\psi^{p}$ o\`u $p$ d\'esigne un entier naturel non nul. 

Supposant que la fonction $\psi$ ne s'annule pas, pour tout entier
naturel $m$ et tout entier rationnel $q$ on d\'efinit la fonction
de classe $\mathcal{C}^{\infty}$ sur le segment $J$:
\[
F_{m}\left(\psi,q\right)=D^{m-1}\left(\psi^{m}D\left(\psi^{q}\right)\right)
\]
Bien entendu $F_{0}\left(\psi,q\right)=\psi^{q}$. 

Il est imm\'ediat de constater, par r\'ecurrence sur l'ordre $q$ de d\'erivation,
la formule:
\begin{equation}\label{eq12}
q\, D^{p}F_{m}\left(\psi,q+p\right)=\left(q+p\right)F_{m+q}\left(\psi,q\right).
\end{equation}
Tenant compte de la formule du produit de Lagrange \eqref{eq:3-1} et de la formule
de d\'erivation \eqref{eq11} ci-dessus, pour tout couple $\left(n,p\right)$
d'entiers naturels v\'erifiant $1\leq p\leq n$ on obtient: 
\begin{equation}
2D^{p}F_{n-p}\left(\psi,-p\right)=\sum_{k=0}^{n}\binom{n}{k}F_{k}\left(\psi,-p\right)F_{n-k}\left(\psi,-p\right)\label{eq:5}
\end{equation}
formule plus g\'en\'erale que celle de Frobenius-Stickelberger,
\'ecrite par ces auteurs dans le cas particulier $p=1$. Pour tout
couple $\left(n,p\right)$ d'entiers non nuls on
peut disposer aussi de la formule : 
\begin{equation}
2F_{n+p}\left(\psi,p\right)=D^{p}\left(\,\sum_{k=0}^{n}\binom{n}{k}F_{k}\left(\psi,p\right)F_{n-k}\left(\psi,p\right)\right)\label{eq:6-1}.
\end{equation}
\section{Sur un th\'eor\`eme de P. J. Olver}
\subsection*{Notations }
On d\'esigne par $\mathcal{A}\left(J\right)$ la $\BC$-alg\`ebre 
des fonctions r\'eelle-analytiques sur $J$, \`a valeurs complexes.
Le\textit{ th\'eor\`eme de Pringsheim }stipule qu'une fonction complexe
$f\in \mathcal{C}^{\infty}\left(J\right)$  est r\'eelle-analytique si et seulement
s' il existe 
$R>0$ tel qu'on ait \[{\displaystyle \sup_{p\in\mathbb{N}}\dfrac{\left\Vert D^{p}g\right\Vert }{p!\, R^{p}}<+\infty}.\]
Un tel nombre sera appel\'e  niveau de la fonction
$g$.

\noindent Pour tout nombre r\'eel $R>0$, on introduit le sous-espace vectoriel
de l'alg\`ebre $\mathcal{A}\left(J\right)$: 
\[
\mathcal{A}_{R}\left(J\right)=\left\{ g\in\mathcal{A}\left(J\right)\;;\;\sup_{p\in\mathbb{N}}\dfrac{\left\Vert D^{p}g\right\Vert }{p!\, R^{p}}<+\infty\right\} 
\]
qu'on munit de la norme ${\displaystyle N_{R}\left(g\right)=\sup_{p\in\mathbb{N}}\dfrac{\left\Vert D^{p}g\right\Vert }{p!\, R^{p}}}$.
L'espace norm\'e $\mathcal{A}_{R}\left(J\right)$ est complet. Pour
tous nombres r\'eels $R$ et $S$ v\'erifiant $0<R<S$, on a l'inclusion
$\mathcal{A}_{R}\left(J\right)\subset\mathcal{A}_{S}\left(J\right)$
et l'injection canonique est continue: pour toute fonction $g$ appartenant
\`a l'espace norm\'e $\mathcal{A}_{R}\left(J\right)$ on a l'in\'egalit\'e
$N_{S}\left(g\right)\leq N_{R}\left(g\right)$. L'alg\`ebre $\mathcal{A}\left(J\right)$
est munie de la structure \textit{limite-inductive de la suite croissante
des sous-espaces de Banach $\left(\mathcal{A}_{k}\left(J\right)\right)_{k\geq1}$}. 

\'Etant donn\'es une suite $\left(a_{n}\right)_{n\in\mathbb{N}}$ d'\'el\'ements
de l'alg\`ebre $\mathcal{A}\left(J\right)$ et $\zeta\in\mathbb{C}^{*}$,
la s\'erie ${\displaystyle \sum_{n=0}^{\infty}\zeta^{n}a_{n}}$ est
convergente dans l'alg\`ebre $\mathcal{A}\left(J\right)$ si, par d\'efinition,
elle l'est dans un espace de Banach $\mathcal{A}_{R}\left(J\right)$,
ce qui suppose en particulier l'appartenance des sommes partielles
\`a l'espace de Banach $\mathcal{A}_{R}\left(J\right)$ ; ainsi toutes les
fonctions coefficients $a_{n}$ ont un m\^eme niveau $R$, car le nombre
complexe $\zeta$ est distinct de $0$. La suite de terme g\'en\'eral
$\zeta^{n}a_{n}$ est donc born\'ee dans l'espace de Banach $\mathcal{A}_{R}\left(J\right)$,
de sorte que l'ensemble 
\[
\bigl\{ r>0\;;\;\mathfrak{A}\left(r\right)=\sup_{n}N_{R}\left(a_{n}\right)r^{n}<+\infty\bigr\}
\]
n'est pas vide puisqu'il contient $\left|\zeta\right|$, et est un
intervalle dont la borne sup\'erieure $\rho$ est strictement positive,
\'eventuellement infinie. Ce nombre $\rho$ est le \textit{rayon de
convergence de la s\'erie enti\`ere} ${\displaystyle \sum_{n=0}^{\infty}z^{n}a_{n}}$: pour tout nombre complexe $z$ v\'erifiant $\left|z\right|<\rho$,
la s\'erie ${\displaystyle \sum_{n=0}^{\infty}z^{n}a_{n}}$ est absolument
convergente dans l'espace de Banach $\mathcal{A}_{R}\left(J\right)$,
de plus on a la formule de Hadamard: 
\[
\rho=\Bigl(\limsup_{n}\left(\sqrt[n]{N_{R}\left(a_{n}\right)}\right)\Bigr)^{-1}>0.
\]
Ainsi on dispose de l'application somme de s\'erie enti\`ere \`a
valeurs dans $\mathcal{A}\left(J\right)$: 
\begin{equation*}
 \overset{\circ}{D}\left(0,\rho\right)\overset{f}{\longrightarrow}\mathcal{A}\left(J\right),\qquad f\left(z\right)={\displaystyle \sum_{n=0}^{\infty}z^{n}a_{n}}
\end{equation*}
qui satisfait, pour tout  $z\in \BC,\;\left|z\right|<\rho$, \`a l'in\'egalit\'e 
\begin{equation*} 
N_{R}\left(f\left(z\right)\right)\leq\sum_{n=0}^{\infty}N_{R}\left(a_{n}\right)\left|z\right|^{n}
\end{equation*}
et la s\'erie de fonctions ${\displaystyle \sum_{n=0}^{\infty}z^{n}a_{n}}$
est normalement convergente sur le segment $J$ et ainsi pour tout point
$t$ du segment $J$ on a $f\left(z\right)\left(t\right)={\displaystyle \sum_{n=0}^{\infty}a_{n}\left(t\right)z^{n}}$.
La condition d'appartenance des coefficients $a_{n}$ \`a l'un des sous-espaces
de Banach $\mathcal{A}_{R}\left(J\right)$ et l'in\'egalit\'e stricte
$\rho>0$ sont suffisantes pour assurer que pour tout $z$ de module
strictement inf\'erieur \`a $\rho$, la somme de la s\'erie de fonctions
${\displaystyle \sum_{n=0}^{\infty}a_{n}\left(t\right)z^{n}}$ de
la variable r\'eelle $t$, est analytique sur le segment $J$. On peut observer 
que la convergence normale sur le segment $J$ de cette s\'erie et de
toutes ses s\'eries d\'eriv\'ees, ne conduit pas \`a l'analyticit\'e de la somme,
comme le montre l'exemple ${\displaystyle \sum_{n=0}^{\infty}\dfrac{1}{1+nt}z^{n}}$
sur le segment $\bigl[0,1\bigr]$. 

On va \'enoncer \`a pr\'esent quelques lemmes qui nous seront utiles par la suite
\begin{lem}\label{t3}
Soit une suite $\left(f_{n}\right)_{n\in\mathbb{N}}$ d'\'el\'ements de
l'alg\`ebre $\mathcal{A}\left(J\right)$. On suppose qu'il existe trois
nombres strictement positifs $R$, $A$ et $S$, un nombre entier
$q\geq1$, tels pour tout entier naturel $n\geq1$ on ait $N_{R}\left(f_{n}\right)\leq A\, n!\, S^{n}$. 

Alors la s\'erie enti\`ere ${\displaystyle \sum_{n=0}^{\infty}\dfrac{1}{n!}\, z^{n}f_{n}}$
d\'efinit l'application $f$ somme de s\'erie enti\`ere \`a valeurs dans\textbf{
}$\mathcal{A}\left(J\right)$, pr\'ecis\'ement dans $\mathcal{A}_{R}\left(J\right)$
\[
\overset{\circ}{D}\left(0,S^{-1}\right)\overset{f}{\longrightarrow}\mathcal{A}\left(J\right),\qquad f\left(z\right)=\sum_{n=0}^{\infty}\dfrac{1}{n!}\, z^{n}f_{n}.
\]
\end{lem}
\begin{proof}
Pour tout entier naturel $n\in \BN^*$  on a $\displaystyle N_{R}\left(\dfrac{1}{n!}\, f_{n}\right)\leq S^{n}A$.
De ce fait l'ensemble $\displaystyle \left\{ r>0\;;\;\sup_{n}N_{R}\left(\dfrac{1}{n!}f_{n}\right)r^{n}<+\infty\right\} $
contient l'intervalle $\bigl]0,S^{-1}\bigr[$ et est contenu dans
l'intervalle $\bigl]0,S^{-1}\bigr]$. Observons qu'on a l'in\'egalit\'e
$\displaystyle N_{R}\left(f\left(z\right)\right)\leq\dfrac{A}{1-S\left|z\right|}$. \end{proof}
\begin{lem}\label{t4}
Pour tout entier naturel $p\geq1$, pour tout $p$-uple $\left(g_{1},\ldots,g_{p}\right)$
de fonctions r\'eelle-analytiques de niveau $R$ sur le segment $J$,
on a l'in\'egalit\'e: 
\[
N_{2R}\left(g_{1}\cdots g_{p}\right)\leq2^{p-1}N_{R}\left(g_{1}\right)\cdots N_{R}\left(g_{p}\right).
\]
\end{lem}
\begin{proof}

En vertu de la formule de Leibniz, pour tout entier naturel $n$,
on a:
\[
D^{n}\left(g_{1}\cdots g_{p}\right)=\sum_{\alpha\in E\left(n,p\right)}\dfrac{n!}{\alpha_{1}!\cdots\alpha_{p}!}D^{\alpha_{1}}g_{1}\cdots D^{\alpha_{p}}g_{p}.
\]
Par suite
\begin{align*}
&\left\Vert D^{n}\left(g_{1}\cdots g_{p}\right)\right\Vert \leq n!\sum_{\alpha\in E\left(n,p\right)}\dfrac{1}{\alpha_{1}!}\left\Vert D^{\alpha_{1}}g_{1}\right\Vert \cdots\dfrac{1}{\alpha_{p}!}\left\Vert D^{\alpha_{p}}g_{p}\right\Vert 
\\
&\leq n!\, N_{R}\left(g_{1}\right)\cdots N_{R}\left(g_{p}\right)\sum_{\alpha\in E\left(n,p\right)}R^{\alpha_{1}+\cdots+\alpha_{p}}= \dbinom{n+p-1}{p-1}\, n!\, R^{n}N_{R}\left(g_{1}\right)\cdots N_{R}\left(g_{p}\right)
\\
&\leq2^{n+p-1}n!\, R^{n}N_{R}\left(g_{1}\right)\cdots N_{R}\left(g_{p}\right)= n!\,\left(2R\right)^{n}2^{p-1}N_{R}\left(g_{1}\right)\cdots N_{R}\left(g_{p}\right).
\end{align*}
\end{proof}
\begin{lem}\label{t5}
Pour toute fonction r\'eelle-analytique $g$ de niveau $S$ sur le segment
$J$, et pour tout entier naturel $q$ on a l'in\'egalit\'e: 
\[
N_{2S}\left(D^{q}g\right)\leq N_{S}\left(g\right)q!\left(2S\right)^{q}.
\]
\end{lem}
\begin{proof}
Comme on a l'in\'egalit\'e $\left(p+q\right)!\leq p!\, q!\,2^{p+q}$,
on a successivement: 
\[
\bigl\Vert D^{p}\left(D^{q}g\right)\bigr\Vert=\bigl\Vert D^{p+q}g\bigr\Vert\leq N_{S}\left(g\right)\left(p+q\right)!\, S^{p+q}\leq\left(N_{S}\left(g\right)q!\left(2S\right)^{q}\right)p!\left(2S\right)^{p}.
\]
\end{proof}
\begin{cor}
Soient un couple $\left(u,g\right)$ de fonctions r\'eelle-analytiques
sur $J$, $u$ \'etant de niveau $R$ et $g$ de niveau $S$. On d\'efinit
le nombre $T=2\,\max\left(R,2\, S\right)$. Alors la suite $\left(D^{n-1}\left(u^{n}Dg\right)\right)_{n\geq1}$
d'\'el\'ements de l'alg\`ebre $\mathcal{A}\left(J\right)$ satisfait l'hypoth\`ese
du lemme \eqref{t3}, \`a savoir pour tout entier naturel $n\geq1$: 
\[
N_{2T}\left(D^{n-1}\left(u^{n}Dg\right)\right)\leq\dfrac{S}{T}N_{S}\left(g\right)n!\left(4N_{R}\left(u\right)T\right)^{n}\leq\dfrac{1}{4}N_{S}\left(g\right)n!\left(4N_{R}\left(u\right)T\right)^{n}.
\]
\end{cor}
\begin{proof}
D'apr\`es le lemme \eqref{t5} la fonction r\'eelle-analytique $Dg$ appartient
\`a $\mathcal{A}_{2S}\left(J\right)$, pr\'ecis\'ement $N_{2S}\left(Dg\right)\leq2SN_{S}\left(g\right)$.
Ensuite, d'apr\`es le lemme \eqref{t4}, on a: 
\[
N_{T}\left(u^{n}Dg\right)\leq2^{n}N_{R}\left(u\right)^{n}N_{2S}\left(Dg\right)\leq2^{n+1}S\, N_{R}\left(u\right)^{n}N_{S}\left(g\right).
\]
Puis \`a nouveau en vertu du lemme \eqref{t5}, pour tout entier naturel $n\geq1$
on conclut: 
\[
N_{2T}\left(D^{n-1}\left(u^{n}Dg\right)\right)\leq N_{T}\left(u^{n}Dg\right)\left(n-1\right)!\left(2\, T\right)^{n-1}\leq\dfrac{S}{T}N_{S}\left(g\right)n!\left(4N_{R}\left(u\right)T\right)^{n}.
\]
\end{proof}
\begin{rem}\label{t6}
On dispose donc de l'application somme de s\'erie enti\`ere \`a valeurs
dans $\mathcal{A}\left(J\right)$, pr\'ecis\'ement dans $\mathcal{A}_{2T}\left(J\right)$:
\[
\overset{\circ}{D}\left(0,\left(4N_{R}\left(u\right)T\right)^{-1}\right)\longrightarrow\mathcal{A}\left(J\right),\qquad z\longmapsto\sum_{n=0}^{\infty}\dfrac{z^{n}}{n!}D^{n-1}\left(u^{n}Dg\right).
\]
\end{rem}
Rappelons qu'une fonction somme de s\'erie enti\`ere \`a coefficients complexes
${\displaystyle f\left(z\right)=\sum_{n=0}^{\infty}c_{n}z^{n}}$,
de rayon de convergence strictement positif $\rho$, est telle que
pour tout nombre complexe $z$ v\'erifiant $\left|z\right|<\rho$, pour
tout nombre r\'eel $r$ appartenant \`a l'intervalle $\bigl[\left|z\right|,\rho\bigr[$
on a l'in\'egalit\'e: 
\[
\sum_{k=0}^{\infty}\dfrac{1}{k!}\left|f^{\left(k\right)}\left(z\right)\right|\left(r-\left|z\right|\right)^{k}\leq\sum_{n=0}^{\infty}\left|c_{n}\right|r^{n}=f^{*}\left(r\right).
\]

\begin{prop}
Soient une fonction $g$ r\'eelle-analytique de niveau $S$ sur le segment
$J$, un point $t$ de $J$ et une s\'erie enti\`ere \`a coefficients complexes
${\displaystyle \psi\left(w\right)=\sum_{n=0}^{\infty}c_{n}w^{n}}$,
de rayon de convergence $\rho >0$. La fonction somme
de s\'erie enti\`ere
\[\zeta\longmapsto\varphi_{\, t}\left(\zeta\right)=\psi\left(\zeta-g\left(t\right)\right)\]
est d\'efinie dans le disque ouvert $\overset{\circ}{D}\left(g\left(t\right),\rho\right)$. 

Pour tout segment $J_{t}$ contenant le point $t$, contenu dans le
segment $J$ et tel que l'image $g\left(J_{t}\right)$ soit contenue
dans le disque ouvert $\overset{\circ}{D}\left(g\left(t\right),\rho\right)$,
la fonction compos\'ee $\varphi_{\, t}\circ g$ est r\'eelle-analytique
de niveau $S_{J_{t}}\left(r\right)$ sur le segment $J_{t}$ o\`u:
\[
S_{J_{t}}\left(r\right)=\left(1+\dfrac{N_{S}\left(g\right)}{r-\left\Vert g-g\left(t\right)\right\Vert _{J_{t}}}\right)S\] 
et de plus 
\[\quad\left\Vert g-g\left(t\right)\right\Vert _{J_{t}}<r<\rho
\]
\end{prop}
\begin{proof}
En vertu de la formule de Fa\`a di Bruno exprim\'ee avec les polyn\^omes
de Bell \cite{F2} (ainsi que les r\'ef\'erences qui s'y trouvent),
pour tout entier naturel $p\geq1$, on peut \'ecrire: 
\[
D^{p}\left(\varphi_{\, t}\circ g\right)=\sum_{i=1}^{p}\left(\varphi_{\, t}^{\left(i\right)}\circ g\right)B_{p,i}\left(Dg,\ldots,D^{p-i+1}g\right).
\]
Compte-tenu de l'expression des polyn\^omes de Bell: 
\[
B_{p,i}=\dfrac{1}{i!}\sum_{\alpha\in F\left(p,i\right)}\dfrac{p!}{\alpha_{1}!\cdots\alpha_{i}!}X_{\alpha_{1}}\cdots X_{\alpha_{i}}
\]
o\`u $F\left(p,i\right)$ est l'ensemble 
\[
\bigl\{\alpha=\left(\alpha_{1},\ldots,\alpha_{i}\right)\in\left(\mathbb{N}^{*}\right)^{i}\;;\;\alpha_{1}+\cdots+\alpha_{i}=p\bigr\}
\]
dont le cardinal est ${\displaystyle \binom{p-1}{i-1}}$. On obtient: 
\[
\left\Vert B_{p,i}\left(Dg,\ldots,D^{p-i+1}g\right)\right\Vert _{J}\leq\binom{p-1}{i-1}\dfrac{1}{i!}\, N_{S}\left(g\right)^{i}\, p!\, S^{p}.
\]
Pour tout $s$ appartenant \`a $J_{t}$ on tire l'in\'egalit\'e: 
\[
\bigl|D^{p}\left(\varphi_{\, t}\circ g\right)\left(s\right)\bigr|
\leq p!\, S^{p}\sum_{j=0}^{p-1}\binom{p-1}{j}\dfrac{1}{\left(j+1\right)!}\bigl|\varphi_{\, t}^{\left(j+1\right)}\left(g\left(s\right)\right)\bigr|N_{S}\left(g\right)^{j+1}.
\]
Or, pour tout $s$ appartenant \`a $J_{t}$ et pour tout nombre $r$
v\'erifiant $\left|g\left(s\right)-g\left(t\right)\right|<r<\rho$,
on dispose des in\'egalit\'es: 
\[
\dfrac{1}{\left(j+1\right)!}\bigl|\varphi_{\, t}^{\left(j+1\right)}\left(g\left(s\right)\right)\bigr|=\dfrac{1}{\left(j+1\right)!}\bigl|\psi^{\left(j+1\right)}\left(g\left(s\right)-g\left(t\right)\right)\bigr|\leq\dfrac{\psi^{*}\left(r\right)}{\left(r-\left|g\left(s\right)-g\left(t\right)\right|\right)^{j+1}}
\]
desquelles on d\'eduit: 
\[
\bigl|D^{p}\left(\varphi_{\, t}\circ g\right)\left(s\right)\bigr|\leq\psi^{*}\left(r\right)p!\, S^{p}\sum_{j=0}^{p-1}\binom{p-1}{j}\dfrac{N_{S}\left(g\right)^{j+1}}{\left(r-\left|g\left(s\right)-g\left(t\right)\right|\right)^{j+1}}
\]
ou
\begin{equation}\label{eq13}
\bigl|D^{p}\left(\varphi_{\, t}\circ g\right)\left(s\right)\bigr|
\leq\psi^{*}\left(r\right)\, p!\, S^{p}\left(1+\dfrac{N_{S}\left(g\right)}{r-\left|g\left(s\right)-g\left(t\right)\right|}\right)^{p}.
\end{equation}
La conclusion s'ensuit imm\'ediatement. 
\end{proof}
\subsection{Formulation du probl\`eme }
Soient un couple $\left(u,g\right)$ de fonctions r\'eelle-analytiques
sur $J$, $u$ \'etant de niveau $R$ et $g$ de niveau $S$, un point
$t$ de $J$ et une s\'erie enti\`ere \`a coefficients complexes ${\displaystyle \psi\left(w\right)=\sum_{n=0}^{\infty}c_{n}w^{n}}$,
de rayon de convergence $\rho >0$, \`a laquelle est
associ\'ee la fonction somme de s\'erie enti\`ere, $\zeta\longmapsto\varphi_{\, t}\left(\zeta\right)=\psi\left(\zeta-g\left(t\right)\right)$,
d\'efinie dans le disque ouvert $\overset{\circ}{D}\left(g\left(t\right),\rho\right)$. 

D'apr\`es la remarque \eqref{t6}, notant $T_{J_{t}}=2\max\left(R,2S_{J_{t}}\left(\rho\right)\right)$,
comme on a $T_{J_{t}}\geq T$, on obtient les applications
somme de s\'erie enti\`ere \`a valeurs dans $\mathcal{A}\left(J_{t}\right)$
\begin{align*}
\overset{\circ}{D}\left(0,\left(4N_{R}\left(u\right)T_{J_{t}}\right)^{-1}\right)\longrightarrow\mathcal{A}\left(J_{t}\right),& \qquad z\longmapsto\sum_{n=0}^{\infty}\dfrac{z^{n}}{n!}D^{n-1}\left(u^{n}Dg\right)
\\
\overset{\circ}{D}\left(0,\left(4N_{R}\left(u\right)T_{J_{t}}\right)^{-1}\right)\longrightarrow\mathcal{A}\left(J_{t}\right),& \qquad z\longmapsto\sum_{n=0}^{\infty}\dfrac{z^{n}}{n!}D^{n-1}
\left(u^{n}D\left(\varphi_{\, t}\circ g\right)\right).
\end{align*}
La formule \eqref{eq11}  sugg\`ere la question: pour quels nombres
complexes $z$ de module \'eventuellement plus petit que $\left(4N_{R}\left(u\right)T_{J_{t}}\right)^{-1}$,
peut-on, pour tout $s$ du segment $J_{t}$, consid\'erer la
compos\'ee $\varphi_{\, t}\left({\displaystyle \sum_{n=0}^{\infty}\dfrac{z^{n}}{n!}D^{n-1}\left(u^{n}Dg\right)\left(s\right)}\right)$
et avoir l'\'egalit\'e$ $
\[
\varphi_{\, t}\left(\sum_{n=0}^{\infty}\dfrac{z^{n}}{n!}D^{n-1}\left(u^{n}Dg\right)\left(s\right)\right)=\sum_{n=0}^{\infty}\dfrac{z^{n}}{n!}D^{n-1}\left(u^{n}D\left(\varphi_{\, t}\circ g\right)\right)\left(s\right).
\]
\subsection{\'Etude et r\'esolution du probl\`eme}
\textit{Afin de simplifier les \'enonc\'es ult\'erieurs on s'attache dor\'enavant
au cas particulier du segment $J_{t}=\left\{ t\right\} $, sans que
cela modifie la substance de l'\'etude. }De ce fait, notons qu'on a
: $S_{\left\{ t\right\} }\left(\rho\right)=\left(1+\dfrac{N_{S}\left(g\right)}{\rho}\right)S=S_{0}$,
on \'ecrit $T_{\left\{ t\right\} }=T_{0}=2\max\left(R,2S_{0}\right)$.
On remarque que si le nombre complexe $z$ satisfait l'in\'egalit\'e $\left|z\right|<\dfrac{1}{4N_{R}\left(u\right)T_{0}}$
 alors
\[\sum_{n= 0}^{\infty}\frac{|z|^n}{n!} |D^{n-1}\left(  u^n Dg  \right)(t)|< \infty\]
et
\[\sum_{n=0}^{\infty}\dfrac{\left|z\right|^{n}}{n!}\bigl|\left(u^{n}D\left(\varphi_{\, t}\circ g\right)\right)\left(t\right)\bigr|<+\infty.\]
On suppose donc que le nombre complexe $z$ satisfait \`a cette  l'in\'egalit\'e.
Devant disposer de l'expression $\varphi_{\, t}\left({\displaystyle \sum_{n=0}^{\infty}\dfrac{z^{n}}{n!}D^{n-1}\left(u^{n}Dg\right)\left(s\right)}\right)$
on souhaite l'in\'egalit\'e 
\[
\left|{\displaystyle \sum_{n=1}^{\infty}\dfrac{z^{n}}{n!}D^{n-1}\left(u^{n}Dg\right)\left(t\right)}\right|<\rho.
\]
Comme on a 
\[
{\displaystyle \sum_{n=0}^{\infty}\dfrac{z^{n}}{n!}D^{n-1}\left(u^{n}Dg\right)\left(t\right)=g\left(t\right)+\sum_{n=1}^{\infty}\dfrac{z^{n}}{n!}D^{n-1}\left(u^{n}Dg\right)\left(t\right)}
\]
l'in\'egalit\'e souhait\'ee est impliqu\'ee par l'in\'egalit\'e ${\displaystyle \sum_{n=1}^{\infty}\dfrac{\left|z\right|^{n}}{n!}\bigl|D^{n-1}\left(u^{n}Dg\right)\left(t\right)\bigr|<\rho}$.
Gr\^ace \`a la formule de Leibniz on a 

\[
\dfrac{1}{n!}\bigl|D^{n-1}\left(u^{n}Dg\right)\left(t\right)\bigr|\leq\dfrac{1}{n}\sum_{\alpha\in E\left(n-1,n+1\right)}\dfrac{1}{\alpha_{1}!}\bigl|D^{\alpha_{1}}u\left(t\right)\bigr|\cdots\dfrac{1}{\alpha_{n}!}\bigl|D^{\alpha_{n}}u\left(t\right)\bigr|\dfrac{1}{\alpha_{n+1}!}\bigl|D^{1+\alpha_{n+1}}g\left(t\right)\bigr|
\]
\[
\leq N_{R}\left(u\right)^{n}N_{S}\left(g\right)R^{n}\sum_{\alpha\in E\left(n-1,n+1\right)}\left(\dfrac{S}{R}\right)^{1+\alpha_{n+1}}.
\]
Mais en g\'en\'eral pour tout nombre r\'eel $x$ strictement positif on
a 
\[
\sum_{\alpha\in E\left(n-1,n+1\right)}x^{1+\alpha_{n+1}}=x\sum_{j=0}^{n-1}\dbinom{2n-2-j}{n-1}\, x^{j}\leq x\sum_{j=0}^{n-1}2^{2n-2-j}x^{j}=2^{2n-1}\sum_{j=1}^{n}\left(\dfrac{x}{2}\right)^{j}
\]
Il vient alors 
\[
\dfrac{1}{n!}\bigl|D^{n-1}\left(u^{n}Dg\right)\left(t\right)\bigr|\leq\dfrac{1}{2}\left(4N_{R}\left(u\right)R\right)^{n}N_{S}\left(g\right)\sum_{j=1}^{n}\left(\dfrac{S}{2R}\right)^{j}.
\]
Par suite  
\[
\sum_{n=0}^{\infty}\dfrac{\left|z\right|^{n}}{n!}\bigl|D^{n-1}\left(u^{n}Dg\right)\left(t\right)\bigr|\leq
\dfrac{N_{S}\left(g\right)}{2}\left(\sum_{j=1}^{\infty}\left(2SN_{R}\left(u\right)\left|z\right|\right)^{j}\right)\left(\sum_{m=0}^{\infty}\left(4RN_{R}\left(u\right)\left|z\right|\right)^{m}\right).
\]
L'hypoth\`ese faite sur le nombre $z$ conduit aux deux in\'egalit\'es : 
\[2SN_{R}\left(u\right)\left|z\right|<\dfrac{1}{8}\left(1+\dfrac{N_{S}\left(g\right)}{\rho}\right)^{-1},\qquad 4RN_{R}\left(u\right)\left|z\right|<\dfrac{1}{2}.\]

De l\`a on tire: 
\[
\sum_{n=0}^{\infty}\dfrac{\left|z\right|^{n}}{n!}\bigl|D^{n-1}\left(u^{n}Dg\right)\left(t\right)\bigr|<\dfrac{N_{S}\left(g\right)}{8}\left(1+\dfrac{N_{S}\left(g\right)}{\rho}\right)^{-1}<\dfrac{\rho}{8}<\rho.
\]
En conclusion, pour tout nombre complexe $z$ satisfaisant l'in\'egalit\'e $\left|z\right|<\dfrac{1}{4N_{R}\left(u\right)T_{0}}$,
se pose alors la question de l'\'egalit\'e: 
\[
\varphi_{\, t}\left(\sum_{n=0}^{\infty}\dfrac{z^{n}}{n!}D^{n-1}\left(u^{n}Dg\right)\left(t\right)\right)=\sum_{n=0}^{\infty}\dfrac{z^{n}}{n!}D^{n-1}\left(u^{n}D\left(\varphi_{\, t}\circ g\right)\right)\left(t\right).
\]
D\'esignons par $P_{m}$ la fonction polyn\^ome somme partielle d'ordre
$m$ de la s\'erie enti\`ere de somme $\varphi_{\, t}$, \`a savoir 
\[
P_{m}\left(\zeta\right)=\sum_{k=0}^{m}c_{k}\left(\zeta-g\left(t\right)\right)^{k}.
\]
En vertu de l'\'egalit\'e \eqref{eq11}, On peut \'ecrire: 
\[
\varphi_{\, t}\left(\sum_{n=0}^{\infty}\dfrac{z^{n}}{n!}D^{n-1}\left(u^{n}Dg\right)\left(t\right)\right)=\lim_{M}\, P_{M}\left(\sum_{n=0}^{\infty}\dfrac{z^{n}}{n!}D^{n-1}\left(u^{n}Dg\right)\left(t\right)\right)
\]
\[
=\lim_{M}\,\sum_{n=0}^{\infty}\dfrac{z^{n}}{n!}D^{n-1}\left(u^{n}D\left(P_{M}\circ g\right)\right)\left(t\right)=
\lim_{M}\left(\lim_{N}\,\sum_{n=0}^{N}\dfrac{z^{n}}{n!}D^{n-1}\left(u^{n}D\left(P_{M}\circ g\right)\right)\left(t\right)\right).
\]
Le lemme suivant montre qu'on peut intervertir les limites.
\begin{lem}\label{t7}
Dans les conditions ci-dessus \'enonc\'ees, uniform\'ement par rapport au
couple $\left(N,z\right)$ appartenant \`a $\mathbb{N}\times\Delta$,
o\`u $\Delta$ d\'esigne le disque ouvert centr\'e en $z=0$ et de rayon
$\dfrac{1}{4N_{R}\left(u\right)T_{0}}$, on a: 
\[
\sum_{n=0}^{N}\dfrac{z^{n}}{n!}D^{n-1}\left(u^{n}D\left(\varphi_{\, t}\circ g\right)\right)\left(t\right)=\lim_{M}\,\sum_{n=0}^{N}\dfrac{z^{n}}{n!}D^{n-1}\left(u^{n}D\left(P_{M}\circ g\right)\right)\left(t\right).
\]
\end{lem}
\begin{proof}
Comme $\left(\varphi_{\, t}-P_{M}\right)\left(g\left(t\right)\right)=0$,
il s'agit de d\'emontrer qu'on a: 
\[
\lim_{M}\left(\sup_{N\geq1\: et\: z\in\Delta}\left|\sum_{n=1}^{N}\dfrac{z^{n}}{n!}D^{n-1}\left(u^{n}D\left(\left(\varphi_{\, t}-P_{M}\right)\circ g\right)\right)\left(t\right)\right|\right)=0.
\]
Tout d'abord on d\'efinit les fonctions $\psi_{M}$ et $\phi_{M}$:
\[
{\displaystyle \psi_{M}\left(w\right)=\sum_{m=M+1}^{\infty}c_{m}w^{m}}\qquad\phi_{M}\left(\zeta\right)=\psi_{M}\left(\zeta-g\left(t\right)\right)=\left(\varphi_{\, t}-P_{M}\right)\left(\zeta\right).
\]
Pour tout $r$ v\'erifiant $0<r<\rho$, on a: 
\[
\dfrac{1}{n!}\bigl|D^{n-1}\left(u^{n}D\left(\phi_{M}\circ g\right)\right)\left(t\right)\bigr|
\]
\[
\leq\dfrac{1}{n}\sum_{\alpha\in E\left(n-1,n+1\right)}\dfrac{1}{\alpha_{1}!}\bigl|D^{\alpha_{1}}u\left(t\right)\bigr|\cdots\dfrac{1}{\alpha_{n}!}
\bigl|D^{\alpha_{n}}u\left(t\right)\bigr|\dfrac{1}{\alpha_{n+1}!}\bigl|D^{1+\alpha_{n+1}}\left(\phi_{M}\circ g\right)\left(t\right)\bigr|.
\]
Or, la fonction $\phi_{M}$ \'etant \`a $\psi_{M}$ ce que la fonction
$\varphi_{\, t}$ est \`a $\psi$, en vertu de l'in\'egalit\'e \eqref{eq13}, pour tout entier
naturel $p$ on a: 
\[
\bigl|D^{p}\left(\phi_{M}\circ g\right)\left(t\right)\bigr|\leq\psi_{M}^{\;*}\left(r\right)p!\, S^{p}\left(1+\dfrac{N_{S}\left(g\right)}{r}\right)^{p}.
\]
Donc: 
\[
\dfrac{1}{n!}\bigl|D^{n-1}\left(u^{n}D\left(\phi_{M}\circ g\right)\right)\left(t\right)\bigr|\leq\psi_{M}^{\;*}
\left(r\right)N_{R}\left(u\right)^{n}R^{n}\sum_{\alpha\in E\left(n-1,n+1\right)}\left(\dfrac{S}{R}\left(1+\dfrac{N_{S}\left(g\right)}{r}\right)\right)^{1+\alpha_{n+1}}.
\]
Or, pour tout nombre r\'eel $x$ strictement positif:
\[\sum_{\alpha\in E\left(n-1,n+1\right)}x^{1+\alpha_{n+1}} \leq 2^{2n-1}\sum_{j=1}^{n}\left(\dfrac{x}{2}\right)^{j}.\]
Il vient alors: 
\[
\dfrac{1}{n!}\bigl|D^{n-1}\left(u^{n}D\left(\phi_{M}\circ g\right)\right)\left(t\right)\bigr|\leq\dfrac{1}{2}\psi_{M}^{\;*}\left(r\right)\left(4N_{R}\left(u\right)R\right)^{n}
\sum_{j=1}^{n}\left(\dfrac{S}{2R}\left(1+\dfrac{N_{S}\left(g\right)}{r}\right)\right)^{j}
\]
puis l'in\'egalit\'e: 
\[
\sum_{n=1}^{N}\dfrac{\left|z\right|^{n}}{n!}\bigl|D^{n-1}\left(u^{n}D\left(\phi_{M}\circ g\right)\right)\left(t\right)\bigr|\leq\dfrac{\psi_{M}^{\;*}\left(r\right)}{2\left(1-4\left|z\right|N_{R}
\left(u\right)R\right)}\sum_{j=1}^{N}\left(2\left|z\right|N_{R}\left(u\right)S\left(1+\dfrac{N_{S}\left(g\right)}{r}\right)\right)^{j}.
\]
L'hypoth\`ese sur le nombre complexe $z$ conduit aux deux in\'egalit\'es: 
\[2\left|z\right|SN_{R}\left(u\right)<\dfrac{1}{8}\left(1+\dfrac{N_{S}\left(g\right)}{\rho}\right)^{-1}, \quad 4RN_{R}\left(u\right)\left|z\right|<\dfrac{1}{2}.\]
De l\`a on tire: 
\[
\sum_{n=0}^{\infty}\dfrac{\left|z\right|^{n}}{n!}\bigl|D^{n-1}\left(u^{n}D\left(\phi_{M}\circ g\right)\right)\left(t\right)\bigr|<\psi_{M}^{\;*}\left(r\right)\sum_{j=1}^{\infty}\dfrac{1}{8^{j}}
\left(1+\dfrac{N_{S}\left(g\right)}{r}\right)^{j}\left(1+\dfrac{N_{S}\left(g\right)}{\rho}\right)^{-j}.
\]
Choisissant convenablement le nombre $r$, par exemple $\dfrac{N_{S}\left(g\right)\rho}{\rho+2N_{S}\left(g\right)}\leq r<\rho$,
on obtient: 
\[
\sup_{N\geq1\: et\: z\in\Delta}\,\sum_{n=1}^{N}\dfrac{\left|z\right|^{n}}{n!}\bigl|D^{n-1}\left(u^{n}D\left(\phi_{M}\circ g\right)\right)\left(t\right)\bigr|\leq\psi_{M}^{\;*}\left(r\right)
\sum_{j=1}^{\infty}\dfrac{1}{4^{j}}=\dfrac{1}{3}\sum_{m=M+1}^{\infty}\left|c_{m}\right|r^{m}.
\]
\end{proof}
\noindent En cons\'equence, les limites ayant \'et\'e interverties, on aboutit \`a:
\[
\varphi_{\, t}\left(\sum_{n=0}^{\infty}\dfrac{z^{n}}{n!}D^{n-1}\left(u^{n}Dg\right)\left(t\right)\right)=\lim_{N}\left(\lim_{M}\,\sum_{n=0}^{N}\dfrac{z^{n}}{n!}D^{n-1}\left(u^{n}D\left(P_{M}\circ g\right)\right)\left(t\right)\right)=
\]
\[
=\lim_{N}\,\sum_{n=0}^{N}\dfrac{z^{n}}{n!}D^{n-1}\left(u^{n}D\left(\varphi_{\, t}\circ g\right)\right)\left(t\right)=\sum_{n=0}^{\infty}\dfrac{z^{n}}{n!}D^{n-1}\left(u^{n}D\left(\varphi_{\, t}\circ g\right)\right)\left(t\right)
\]
de sorte qu'on \'enonce le 
\begin{thm}\label{t8}
Soient un couple $\left(u,g\right)$ de fonctions r\'eelle-analytiques
sur le segment $J$, $u$ de niveau $R$, $g$ de niveau $S$, et
une s\'erie enti\`ere \`a coefficients complexes ${\displaystyle \psi\left(w\right)=\sum_{n=0}^{\infty}c_{n}w^{n}}$,
de rayon de convergence $\rho$ strictement positif. La fonction somme
de s\'erie enti\`ere, 
\[\zeta\longmapsto\varphi_{\, t}\left(\zeta\right)=\psi\left(\zeta-g\left(t\right)\right)\]
est d\'efinie dans le disque ouvert $\overset{\circ}{D}\left(g\left(t\right),\rho\right)$.
On note \[S_{0}=\left(1+\dfrac{N_{S}\left(g\right)}{\rho}\right)S,\quad
T_{0}=2\max\left(R,2S_{0}\right).\]
Pour tout point $t$ du segment $J$, et pour tout nombre complexe
$z$ v\'erifiant $\left|z\right|<\dfrac{1}{4\, N_{R}\left(u\right)T_{0}}$
le nombre complexe ${\displaystyle \sum_{n=0}^{\infty}\dfrac{1}{n!}D^{n-1}\left(u^{n}Dg\right)\left(t\right)z^{n}}$
appartient au disque ouvert $\overset{\circ}{D}\left(g\left(t\right),\rho\right)$
et on a la relation de commutation
\begin{equation}
\varphi_{\, t}\left(\,\sum_{n=0}^{\infty}\dfrac{1}{n!}D^{n-1}\left(u^{n}Dg\right)\left(t\right)z^{n}\right)=\sum_{n=0}^{\infty}\dfrac{1}{n!}D^{n-1}\left(u^{n}D\left(\varphi_{\, t}
\circ g\right)\right)\left(t\right)z^{n}.\label{eq:1-4}
\end{equation}
\end{thm}
\begin{cor}\label{t9}
Soit un couple $\left(u,g\right)$ de fonctions r\'eelle-analytiques
sur le segment $J$, $u$ de niveau $R$, $g$ de niveau $S$, et
on note $T=2\max\left(R,2S\right)$.\textup{ }Soit une s\'erie enti\`ere
\`a coefficients complexes ${\displaystyle \phi\left(w\right)=\sum_{n=0}^{\infty}c_{n}w^{n}}$,
de rayon de convergence infini. 

Pour tout point $t$ du segment $J$, et pour tout nombre complexe
$z$ v\'erifiant $\left|z\right|<\dfrac{1}{4\, N_{R}\left(u\right)T}$
on a la relation: 
\begin{equation}
\phi\left(\,\sum_{n=0}^{\infty}\dfrac{1}{n!}D^{n-1}\left(u^{n}Dg\right)\left(t\right)z^{n}\right)=\sum_{n=0}^{\infty}\dfrac{1}{n!}D^{n-1}\left(u^{n}D\left(\phi\circ g\right)\right)\left(t\right)z^{n}.\label{eq:2-2}
\end{equation}
\end{cor}
\begin{rem}
Dans le contexte pr\'esent le r\'esultat de P. J. Olver \cite{F4}
s'\'ecrit bri\`evement comme suit. Pour tout point $t$ du segment $J$,
et pour tout nombre complexe $z$ v\'erifiant $\left|z\right|<\dfrac{1}{16\, N_{R}\left(u\right)R}$
on a la relation 
\[
\phi\left(\,\sum_{n=0}^{\infty}\dfrac{1}{n!}D^{n-1}\left(u^{n}Du\right)\left(t\right)z^{n}\right)=\sum_{n=0}^{\infty}\dfrac{1}{n!}D^{n-1}\left(u^{n}D\left(\phi\circ u\right)\right)\left(t\right)z^{n}.
\]
\end{rem}
Soient  $\left(u,g\right)$ un couple de fonctions r\'eelle-analytiques
sur le segment $J$, $u$ de niveau $R$, $g$ de niveau $S$. On
peut d\'esigner par $L_{u}\left(g\right)$ l'application somme de la
s\'erie enti\`ere\textbf{\textit{ }}\`a valeurs dans $\mathcal{A}\left(J\right)$
\[
\sum_{n=0}^{\infty}\dfrac{z^{n}}{n!}D^{n-1}\left(u^{n}Dg\right)
\]
dont le rayon de convergence est, selon le corollaire \eqref{t9}, sup\'erieur
ou \'egal \`a $\dfrac{1}{4\, N_{R}\left(u\right)T}$. Dans la suite on
pr\'ef\`ere adopter l'\'ecriture $L_{u}\left(g,z\right)$ au lieu de $L_{u}\left(g\right)\left(z\right)$.
Le r\'esultat qui suit n'est autre que la \guillemotleft formule d'inversion
de Lagrange dans le cadre des fonctions r\'eelle-analytiques \guillemotright,
obtenue donc \`a partir de la formule du produit de Lagrange, de l'extension
du r\'esultat de Olver, et sans passer par la formule int\'egrale de Cauchy.
De plus, notons que la \guillemotleft formule d'inversion de Lagrange dans le
cadre des fonctions holomorphes \guillemotright peut se d\'eduire assez facilement
de la proposition qui suit. 
\begin{prop}
\begin{enumerate}
 \item 
Soit un couple $\left(u,g\right)$ de fonctions r\'eelle-analytiques
sur le segment $J$, $u$ de niveau $R$, $g$ de niveau $S$. On
note $T=2\max\left(R,2S\right)$. Pour tout point $t$ du segment $J$
et pour tout nombre complexe $z$ v\'erifiant $\left|z\right|<\dfrac{1}{4\, N_{R}\left(u\right)T}$,
si le nombre complexe $t+z\, L_{u}\left(u,z\right)\left(t\right)$
appartient au segment $J$, alors on a: 
\[
L_{u}\left(g,z\right)\left(t\right)=g\left(t+z\, L_{u}\left(u,z\right)\left(t\right)\right).
\]
\item Pour tout point $t$ du segment $J$ et pour tout nombre
r\'eel $z$ v\'erifiant $\left|z\right|<\dfrac{1}{16\, N_{R}\left(u\right)R}$,
si le nombre complexe $t+z\, L_{u}\left(u,z\right)\left(t\right)$
appartient au segment $J$, alors on a: 
\[
L_{u}\left(u,z\right)\left(t\right)=u\left(t+z\, L_{u}\left(u,z\right)\left(t\right)\right).
\]
Autrement dit, si le nombre $x=t+z\, L_{u}\left(u,z\right)\left(t\right)$
appartient au segment $J$, alors il satisfait l'\'equation $t+z\, u\left(x\right)=x$.
La formule 
\[
x=t+\sum_{n=1}^{\infty}\dfrac{z^{n}}{n!}D^{n-1}\left(u^{n}\right)\left(t\right)
\]
est la formule d'inversion de Lagrange pour l'\'equation $t+z\, u\left(x\right)=x$
et relativement aux fonctions r\'eelle-analytiques. De plus, pour tout
nombre r\'eel $z$ v\'erifiant $\left|z\right|<\dfrac{1}{8\, N_{R}\left(u\right)T}$,
l'expression de $g\left(x\right)$, o\`u le nombre $x=t+z\, L_{u}\left(u,z\right)\left(t\right)$
est solution r\'eelle de l'\'equation $t+z\, u\left(x\right)=x$, est:
\[
g\left(x\right)=g\left(t\right)+\sum_{n=1}^{\infty}\dfrac{z^{n}}{n!}D^{n-1}\left(u^{n}Dg\right)\left(t\right)
\]
formule qui compl\`ete la formule d'inversion de Lagrange.
\end{enumerate}
\end{prop}
\begin{proof}
\noindent \begin{enumerate}
 \item Supposant le nombre complexe $z$ v\'erifier $\displaystyle \left|z\right|<\dfrac{1}{4\, TN_{R}\left(u\right)}$,
on transforme l'expression de $L_{u}\left(g,z\right)$ ainsi: 
\[
L_{u}\left(g,z\right)=g+\sum_{n=1}^{\infty}\left(\,\sum_{k=0}^{n-1}\dfrac{z^{n}}{n!}\binom{n-1}{k}\, D^{n-k-1}\left(u^{n}\right)D^{k+1}g\right).
\]
Afin de justifier l'interversion dans $\mathcal{A}\left(J\right)$
des sommes ci-dessus, on effectue les majorations suivantes: 
\begin{align*}
{}&\sum_{n=1}^{\infty}\sum_{k=0}^{n-1}\dfrac{\left|z\right|^{n}}{n!}\binom{n-1}{k}\, N_{2T}\left(D^{n-k-1}\left(u^{n}\right)\right)N_{2T}\left(D^{k+1}g\right)\\
&\leq\sum_{n=1}^{\infty}\sum_{k=0}^{n-1}\dfrac{\left|z\right|^{n}}{n}\left(k+1\right)\left(2T\right)^{n-k-1}N_{T}\left(u^{n}\right)\left(2T\right)^{k+1}N_{T}\left(g\right)\\
&\leq\sum_{n=1}^{\infty}\sum_{k=0}^{n-1}\dfrac{\left|z\right|^{n}}{n}\left(k+1\right)\left(2T\right)^{n}2^{n-1}\left(N_{\frac{T}{2}}\left(u\right)\right)^{n}N_{T}\left(g\right).
\end{align*}
 Comme on a simultan\'ement $T\geq2R$ et $T\geq4S>S$, on aboutit \`a:
 \begin{align*}
&\sum_{n=1}^{\infty}\sum_{k=0}^{n-1}\dfrac{\left|z\right|^{n}}{n!}\binom{n-1}{k}\, N_{2T}\left(D^{n-k-1}\left(u^{n}\right)\right)N_{2T}\left(D^{k+1}g\right)
\\
&\leq\dfrac{N_{S}\left(g\right)}{2}\sum_{n=1}^{\infty}\dfrac{\left|z\right|^{n}}{n}\left(4\, T\, N_{R}\left(u\right)\right)^{n}\left(\,\sum_{k=0}^{n-1}\left(k+1\right)\right)\\&= \dfrac{N_{S}\left(g\right)}{4}\sum_{n=1}^{\infty}\left(n+1\right)\left|z\right|^{n}\left(4TN_{R}\left(u\right)\right)^{n}<+\infty.
\end{align*}
L'hypoth\`ese $\left|z\right|<\dfrac{1}{4\, N_{R}\left(u\right)T}$
permet donc l'interversion, et on obtient: 
\begin{align*}
L_{u}\left(g,z\right)&= g+\sum_{k=0}^{\infty}\dfrac{1}{k!}D^{k+1}g\left(\,\sum_{n=k+1}^{\infty}\dfrac{z^{n}}{\left(n-k-1\right)!\, n}D^{n-k-1}\left(u^{n}\right)\right)
\\
&= g+\sum_{k=0}^{\infty}\dfrac{z^{k+1}}{k!}D^{k+1}g\left(\,\sum_{j=0}^{\infty}\dfrac{z^{n}}{j!}D^{j}\left(\dfrac{u^{j+k+1}}{j+k+1}\right)\right)
\\
&= g+\sum_{k=0}^{\infty}\dfrac{z^{k+1}}{k!}D^{k+1}g\left(\,\sum_{j=0}^{\infty}\dfrac{z^{n}}{j!}D^{j-1}\left(u^{j}D\left(\dfrac{u^{k+1}}{k+1}\right)\right)\right)
\\
&= g+\sum_{m=1}^{\infty}\dfrac{z^{m}}{m!}D^{m}g\left(\,\sum_{j=0}^{\infty}\dfrac{z^{n}}{j!}D^{j-1}\left(u^{j}D\left(u^{m}\right)\right)\right)=\sum_{m=0}^{\infty}L_{u}\left(u^{m},z\right)\dfrac{z^{m}}{m!}D^{m}g.
\end{align*}
En vertu du corollaire \eqref{t9}, comme on a $\left|z\right|<\dfrac{1}{4\, N_{R}\left(u\right)T}\leq\dfrac{1}{8\, N_{R}\left(u\right)R}$,
pour tout entier naturel $m$ on obtient $L_{u}\left(u^{m},z\right)=\left(L_{u}\left(u,z\right)\right)^{m}$.
De sorte que pour tout $z$ v\'erifiant $\left|z\right|<\dfrac{1}{4\, N_{R}\left(u\right)T}$
et pour tout point $t$ du segment $J$: 
\[
L_{u}\left(g,z\right)\left(t\right)=\sum_{m=0}^{\infty}\dfrac{1}{m!}\left(z\, L_{u}\left(u,z\right)\left(t\right)\right)^{m}D^{m}g\left(t\right).
\]
Notant $\rho_{g}$ le rayon de convergence de la s\'erie enti\`ere ${\displaystyle \sum_{m=0}^{\infty}\dfrac{w^{m}}{m!}D^{m}g\left(t\right)}$.
On constate donc que, pour tout $z$ v\'erifiant $\displaystyle \left|z\right|<\dfrac{1}{4\, N_{R}\left(u\right)T}$
et pour tout point $t$ du segment $J$, on a $\displaystyle \left|z\, L_{u}\left(u,z\right)\left(t\right)\right|\leq\rho_{g}$.
Or, la fonction $z\longmapsto z\, L_{u}\left(u,z\right)\left(t\right)$,
qui est somme d'une s\'erie enti\`ere \`a coefficients complexes, est une
application ouverte, d'o\`u l'in\'egalit\'e stricte $\left|z\, L_{u}\left(u,z\right)\left(t\right)\right|<\rho_{g}$
pour tous $z$ et $t$ comme ci-dessus. En cons\'equence, pour tout
$z$ v\'erifiant $\displaystyle \left|z\right|<\dfrac{1}{4\, N_{R}\left(u\right)T}$
et pour tout point $t$ du segment $J$, et si le nombre complexe
$\displaystyle t+z\, L_{u}\left(u,z\right)\left(t\right)$ appartient au segment
$J$, on obtient: 
\[
L_{u}\left(g,z\right)\left(t\right)=\sum_{m=0}^{\infty}\dfrac{1}{m!}\Bigl(\left(t+z\, L_{u}\left(u,z\right)\left(t\right)\right)-t\Bigr)^{m}D^{m}g\left(t\right)=g\left(t+z\, L_{u}\left(u,z\right)\left(t\right)\right).
\]
\item Si $g=u$, on a $T=4R$. Remarquant l'\'egalit\'e $\displaystyle z\, u\left(x\right)=z\, L_{u}\left(u,z\right)\left(t\right)=x-t$,
on obtient: 
\[
x=t+z\, L_{u}\left(u,z\right)\left(t\right)=t+\sum_{n=0}^{\infty}\dfrac{z^{n+1}}{n!}D^{n-1}\left(u^{n}Du\right)\left(t\right)=t+\sum_{n=1}^{\infty}\dfrac{z^{n}}{n!}D^{n-1}\left(u^{n}\right)\left(t\right)
\]
Les conditions $\displaystyle \left|z\right|<\dfrac{1}{4\, N_{R}\left(u\right)T}$
et $\displaystyle \left|z\right|<\dfrac{1}{16\, N_{R}\left(u\right)R}$ sont simultan\'ement
r\'ealis\'ees d\`es qu'on a $\left|z\right|<\dfrac{1}{8\, N_{R}\left(u\right)T}$,
de sorte qu'on peut appliquer la premi\`ere partie pour, \`a la fois tenir
une solution $x=t+z\, L_{u}\left(u,z\right)\left(t\right)$ de l'\'equation
$t+z\, u\left(x\right)=x$ lorsque ce nombre $x$ appartient au segment
$J$, et obtenir une expression de $g\left(x\right)$, \`a savoir:
\begin{align*}
g\left(t+zL_{u}\left(u,z\right)\left(t\right)\right)&= L_{u}\left(g,z\right)\left(t\right)\\
&=\sum_{n=0}^{\infty}\dfrac{z^{n}}{n!}D^{n-1}\left(u^{n}Dg\right)\left(t\right)=g\left(t\right)+\sum_{n=1}^{\infty}\dfrac{z^{n}}{n!}D^{n-1}\left(u^{n}Dg\right)\left(t\right).
\end{align*}
\end{enumerate}
\end{proof}
\begin{rem}
Rappelons que les s\'eries enti\`eres \`a coefficients complexes ${\displaystyle \phi\left(w\right)=\sum_{k=0}^{\infty}c_{k}w^{k}}$,
de rayon de convergence infini op\`erent par composition sur les s\'eries
formelles ${\displaystyle T=\sum_{n=0}^{\infty}f_{n}X^{n}}$ \`a coefficients
dans l'anneau $\mathcal{C}^{\infty}\left(J\right)$ selon la formule: 
\[
\phi\left(T\right)=\sum_{n=0}^{\infty}\left(\sum_{k=0}^{\infty}c_{k}\left(\dfrac{1}{n!}D^{n}T^{k}\right)\left(0\right)\right)X^{n}.
\]
La formule \eqref{eq11} du corollaire \eqref{t10}, cons\'equence directe de la formule
du produit de Lagrange \eqref{eq10}, et quelques consid\'erations relativement
\'el\'ementaires de convergence dans l'espace de Fr\'echet $\mathcal{C}^{\infty}\left(J\right)$
montrent que pour tout couple $\left(\psi,f\right)$ de fonctions
de classe $\mathcal{C}^{\infty}$ sur le segment $J$ et pour toute
fonction $\phi$ somme de s\'erie enti\`ere \`a coefficients complexes de
rayon de convergence infini, on a la formule: 
\[
\phi\left(\sum_{n=0}^{\infty}\dfrac{1}{n!}D^{n-1}\left(\psi^{n}Df\right)X^{n}\right)=\sum_{n=0}^{\infty}\dfrac{1}{n!}D^{n-1}\left(\psi^{n}D\left(\phi\circ f\right)\right)X^{n}.
\]
\end{rem}

$\,$
\end{document}